\documentclass[12pt, leqno]{article}

\usepackage{amsmath,amssymb,amsthm, epsfig}

\usepackage{hyperref}

\usepackage{amsmath}

\usepackage{amssymb}

\usepackage{color}

\usepackage{ulem}

\usepackage{epsfig}

\usepackage[mathscr]{eucal}

\usepackage{mathptmx} % rm & math
\usepackage[scaled=0.90]{helvet} % ss
\usepackage{courier} % tt
\normalfont
\usepackage[T1]{fontenc}

\title{Regularity for degenerate two-phase \\ free boundary problems}

\author{Raimundo Leit\~ao,  \quad Olivaine S. de Queiroz,  \quad $\&$  \quad Eduardo V. Teixeira }

\date{}
\newlength{\hchng}
\newlength{\vchng}
\def \div {\mathrm{div}}
\def \dist {\mathrm{dist}}

\newtheorem{theorem}{Theorem}[section]

\newtheorem{lemma}[theorem]{Lemma}

\newtheorem{proposition}[theorem]{Proposition}

\theoremstyle{definition}

\theoremstyle{remark}

\newtheorem{remark}[theorem]{Remark}

\numberwithin{equation}{section}

\newcommand{\intav}[1]{\mathchoice {\mathop{\vrule width 6pt height 3 pt depth  -2.5pt
\kern -8pt \intop}\nolimits_{\kern -6pt#1}} {\mathop{\vrule width
5pt height 3  pt depth -2.6pt \kern -6pt \intop}\nolimits_{#1}}
{\mathop{\vrule width 5pt height 3 pt depth -2.6pt \kern -6pt
\intop}\nolimits_{#1}} {\mathop{\vrule width 5pt height 3 pt depth
-2.6pt \kern -6pt \intop}\nolimits_{#1}}}

\begin{document}
\maketitle

\begin{abstract}
%\vspace{1pc}

We provide a rather complete description of the sharp regularity
theory for a family of heterogeneous, two-phase variational free boundary
problems, $\mathcal{J}_\gamma \to $ min,  ruled by nonlinear, $p$-degenerate elliptic operators. Included
in such family are heterogeneous cavitation problems of
Prandtl-Batchelor type; singular degenerate elliptic equations; and
obstacle type systems. The Euler-Lagrange equation associated to $\mathcal{J}_\gamma$ becomes singular along the free interface $\{u= 0\}$. The degree of singularity is, in turn, dimed by the parameter $\gamma \in [0,1]$.  For $0< \gamma < 1$ we show local minima is locally of class $C^{1,\alpha}$ for a sharp $\alpha$ that depends on dimension, $p$ and $\gamma$. For $\gamma = 0$ we obtain a quantitative, asymptotically optimal result, which assures that local minima are Log-Lipschitz continuous. 
The results proven in this article are new even in the classical context of linear, nondegenerate
equations.

\medskip

\noindent{\sc keywords}: free
boundary problems, degenerate elliptic operators, regularity theory.

\smallskip


\tableofcontents

\end{abstract}
%%%%

\section{Introduction}

Let $\Omega\subset\mathbb R^n$ be a bounded domain, $2\le p< +\infty$,
$f \in L^q(\Omega)$ for $q \ge  n$ and $\varphi \in
W^{1,p}(\Omega)\cap L^\infty(\Omega)$, with, say, $  \varphi^{+}  \not = 0$. The objective of the present manuscript is to derive optimal interior
regularity estimates for the archetypal class of heterogeneous
non-differentiable functionals
\begin{equation}\label{eq.variational}
    \mathcal{J}_\gamma(v):=\int_\Omega\left (|\nabla
    v|^p+F_\gamma(v) + f(X)\cdot v \right )dX \longrightarrow \text{min},
\end{equation}
among competing functions $v \in W^{1,p}_0(\Omega) + \varphi$. The
parameter $\gamma$ in \eqref{eq.variational} varies continuously
from $0$ to $1$, i.e., $\gamma \in [0, 1]$ and the
non-differentiable potential $F_\gamma$ is given by
\begin{equation}\label{def Fgamma}
    F_\gamma(v):=\lambda_+(v^+)^\gamma+\lambda_-(v^-)^\gamma,
\end{equation}
for scalars $0 \le \lambda_{-} < \lambda_{+} < \infty$. As usual,
$v^{\pm}:=\max\{\pm v,0\}$, and, by convention,
\begin{equation}\label{def F0}
    F_0(v):=\lambda_+ \chi_{\{v > 0 \}} +\lambda_{-}\chi_{\{v \le 0
    \}}.
\end{equation}

\medskip

The non-differentiability of the potential $F_\gamma$ impels the Euler-Lagrange equation associated to
$\mathcal{J}_\gamma$ to be singular along the \textit{a priori}
unknown interface
$$
    \mathfrak{F}_\gamma := \left ( \partial \{u_\gamma > 0\} \cup \partial \{u_\gamma < 0\}  \right ) \cap \Omega,
$$
between the positive and negative phases of a minimum. In fact,
a minimizer satisfies, in some weak sense, the following
$p$-degenerate and singular PDE
\begin{equation} \label{intro eq satisfied}
    \Delta_p u =\frac{\gamma}{p}\left ( \lambda_+(u^+)^{\gamma-1}\chi_{\{u>0\}}-
    \lambda_-(u^-)^{\gamma-1}\chi_{\{u < 0\}}\right ) + \frac{1}{p} f(X) \quad \text{in}
    \quad\Omega,
\end{equation}
where $\Delta_p u$ denotes the classical $p$-Laplacian operator, 
$$
	\Delta_p u := \text{div}(|\nabla u|^{p-2}\nabla u).
$$
The potential $F_0$ is actually discontinuous and that further enforces the flux balance
\begin{equation}\label{intro flux balance}
    |\nabla u_0^{+}|^{p} - |\nabla u_0^{-}|^p = \dfrac{1}{p-1} \left (\lambda_+ - \lambda_{-} \right ),
\end{equation}
along the free boundary of the problem, which breaks down the continuity of the gradient through $\mathfrak{F}_0$.

\medskip

A number of important mathematical problems, coming from
several different contexts, are modeled by optimization setups, for
which equation \eqref{eq.variational} serves as an em\-ble\-ma\-tic,
leading prototype. This fact has fostered massive investigations, and
linear versions, $p=2$, of the minimization problem
\eqref{eq.variational} have indeed received overwhelming attention
in the past four decades. The upper case $\gamma = 1$ is related to
obstacle type problems. The linear, homogeneous, one phase obstacle
problem, i.e., $p=2$, $f(X) \equiv 0$ and $\varphi \ge 0$ was fully
studied in the 70's by a number of leading mathematicians: Frehse,
Stampacchia, Kinderlehrer, Brezis, Caffarelli, among others. It has
been established that the minimum is locally of class $C^{1,1}$ and
this is the optimal regularity for solution. The two-phase
version of the problem, i.e., with no sign constrain on the boundary
datum $\varphi$, challenged the community for over three decades.
$C^{1,1}$ estimate for two-phase obstacle problems was established
in \cite{S} with the aid of the powerful \textit{almost}
monotonicity formula obtained in \cite{CJK}.

\medskip

The lower limiting case, $\gamma = 0$, relates to jets flow and
cavities problems. The linear, homogeneous, one phase version of
the problem was studied in \cite{AC}, where it is proven that
minima are Lipschitz continuous. The two-phase version of this
problem brings major new difficulties and $C^{0,1}$ local
regularity of minima was proven in \cite{ACF}, with the aid of the
revolutionary Alt-Caffarelli-Friedman monotonicity formula,
developed in that very same article. Gradient estimates for
two-phase cavitation type problem with bounded non-homogeneity, i.e., $p=2$, $f
\in L^\infty$,  $\gamma =0$ in \eqref{eq.variational}, was
established by Caffarelli, Jerison and Kenig with the aid of their
powerful \textit{almost} monotonicity formula, \cite{CJK}.

\medskip

The intermediary problem $0< \gamma < 1$ has also received great
attention in the past decades. The related free boundary problem
can be used, for example, to model the density of certain chemical
specie, in reaction with a porous catalyst pellet.  The linear,
$p=2$, one-phase, $\varphi\ge 0$, homogeneous, $f\equiv 0$,
version of the problem \eqref{eq.variational} is the theme of a
successful program developed in the 80's by Phillips and
Alt-Phillips, \cite{phillips-cpde}, \cite{phillips-iumj} and
\cite{alt-phillips-crelle}, among others. In similar setting,
H\"older continuity of the gradient of minimizers was proven
Giaquinta and Giusti \cite{GG}.  Further investigations on the
linear, two-phase version of this problem also require powerful
monotonicity formulae in their studies, see \cite{weiss-cpde}.

\medskip

In the mathematical analysis of variational free boundary problems
as \eqref{eq.variational}, the first major key issue to be
addressed concerns the optimal regularity estimate available for a
given minimum. A simple inference on the weak Euler-Lagrange
equation satisfied by a minimum, Equation \eqref{intro eq
satisfied} and also the Flux Balance \eqref{intro flux balance}
for $\gamma =0$, revel that $\Delta_p u$ blows-up along the free
boundary of the problem, $\mathfrak{F}_\gamma := \partial
\{u_\gamma > 0\} \cup \partial \{u_\gamma <0\}$. Therefore, it
becomes a fundamental question to understand precisely how this
phenomenon affects the (lack of) smoothness properties of minima.
Under such perspective, and to some extent, the theory of
two-phase free boundary problems governed by non-linear,
degenerate elliptic operators had hitherto been unaccessible
through current literature, mainly due to the lack of monotonicity
formulae in this context.

\medskip

In the study of sharp smoothness properties of minima to the
functional $\mathcal{J}_\gamma$, further difficulties also arise
from the very complexity of the regularity theory for the
governing operator $\Delta_p$. We recall that $p$-harmonic functions,
i.e., solutions to the homogeneous equation
$$
    \Delta_p h = 0 \quad \text{ in } B_1,
$$
are locally of class $C^{1,\alpha_p}$ for an exponent $0 <
\alpha_p < 1$ that depends only upon dimension and $p$. The
precise value of $\alpha_p$ is in general unknown -- see
\cite{iwa-man-rmi} for the planar case $n=2$.  This fact indicates
that interior estimates available for $p$-harmonic functions, that
in turn are below quadratic, $C^{1,1}$, will compete with optimal
growth along the free interface $\mathfrak{F}_\gamma$. The
regularity theory for heterogeneous equations $\Delta_p \xi =
f(X)$ is even further involved and, up to our knowledge, the
understanding on this class of problems is not yet fully complete.

\medskip

From the mathematical point of view, the exponent $\gamma$ appearing
in \eqref{eq.variational} should be comprehended as the parameter
that measures the singularity of the absorption term of the related
equation. For non-differentiable but continuous functionals,
$\mathcal{J}_\gamma$ with $0 < \gamma \le 1$, it has been
conjectured that the gradient of a minimum is locally continuous, even through the singular free interface
$\mathfrak{F}_\gamma$. The first result we present in this paper
gives an affirmative answer to such question. Furthermore, it
provides the \textit{asymptotically} optimal $C^{1,\alpha}$ interior
regularity theory available for minima of such functionals.

\begin{theorem}[$C^{1,\alpha}$ regularity estimates] \label{thm 1} Let $u$ be a minimizer of the problem \eqref{eq.variational}. Assume $0< \gamma \le 1$ and $f \in L^q(\Omega)$, for some $q > n$. Then $u \in C_\text{loc}^{1,\alpha}$, for
\begin{equation}\label{sharp reg intro}
    \alpha := \min \left \{ \alpha_p^{-}, \frac{\gamma}{p-\gamma},  \frac{(q-n)}{(p-1)q}  \right \},
\end{equation}
where the estimate indicated in \eqref{sharp reg intro} should be read as
\begin{equation}\label{sharp reg intro explained}
    \left |
        \begin{array}{lll}
            \text{If } \min \left \{\frac{\gamma}{p-\gamma}, \frac{(q-n)}{(p-1)q} \right \} < \alpha_p, &\text{then }  & u \in C^{1,\min \left \{\frac{\gamma}{p-\gamma}, \frac{(q-n)}{(p-1)q} \right \}}. \\ \\
            \text{If }  \min \left \{\frac{\gamma}{p-\gamma}, \frac{(q-n)}{(p-1)q} \right \}  \ge \alpha_p, &\text{then } & u \in C^{1,\sigma}, \text{ for any } 0<\sigma < \alpha_p.
    \end{array}
  \right.
\end{equation}
Furthermore, for any $\Omega' \Subset \Omega$, there exists a
constant $C>0$ depending only on, $\Omega'$, n, $p$, $q$,
$\|\varphi\|_{L^\infty(\Omega)}$, $\|f\|_{L^q(\Omega)}$,
$\lambda_{+}$, $\lambda_{-}$, $\gamma$ and $(\alpha_p - \alpha)$, such that
$$
    \|u\|_{C^{1,\alpha}(\Omega')} \le C.
$$
\end{theorem}

\medskip

Before continuing, let us make few comments on Theorem \ref{thm 1}
and its implications. The key ingredient of  the regularity
estimate established in Theorem \ref{thm 1} reveals how the
competing forces involved in the lack of smoothness for minima of \eqref{eq.variational}, namely
$$
(\text{regularity theory for } \Delta_p ) \times (\text{singular
absorption term } \sim u^{\gamma-1} ) \times (\text{roughness of
the source } f)
$$
get adjusted, via the sharp relation \eqref{sharp reg intro}.
Regarding the exponent $\alpha_p$, one easily verifies that  the function  $X \mapsto |X|^{\frac{p}{p-1}}$ has bounded $p$-laplacian, thus $\alpha_p$ is appearing in \eqref{sharp reg intro} is 
below the critical value $\frac{1}{p-1}$. However, nonnegative
functions with bounded $p$-Laplacian, $v$, do grow as
$\text{dist}^{\frac{p}{p-1}}(X, \mathfrak{F})$ away from
$\mathfrak{F} = \partial \{v > 0 \}$, see \cite{KKPS}.  In this
particular setting, it is possible to replace $\alpha_p$ by
$\frac{1}{p-1}$ in \eqref{sharp reg intro}. Thus, at least if $f \in
L^\infty$, Theorem \ref{thm 1} revels $u \in
C^{\frac{\gamma}{p-\gamma}}$, which is the precise generalization of
the optimal regularity estimate obtained for the one-phase linear
setting $p=2$, see for instance \cite{phillips-iumj, phillips-cpde}.

\medskip

Confronting the effect of the singular absorption term $\sim
u^{\gamma -1}$ and the influence of integrability properties of
the source $f$, we conclude that solutions to
\eqref{eq.variational} are locally in $C^{1,  \min\{ \alpha_p^{-},
\frac{\gamma}{p-\gamma} \}}$, provided $f \in L^q$ for any
\begin{equation}\label{best q}
    q \ge n \cdot \frac{(p-\gamma)}{p(1-\gamma)} =: q(p,n,\gamma).
\end{equation}
Interestingly enough, one verifies that
\begin{equation}\label{int of f gamma=1}
    q(p,n,1^{-}) = \infty \quad \text{ and } \quad \lim\limits_{\gamma \to 1} \frac{(q(p,n,\gamma)-n)p}{(p-1)q(p,n,\gamma)} = \frac{1}{p-1}.
\end{equation}
Also it is revealing to compute the limit
\begin{equation}\label{int of f gamma=0}
    \lim\limits_{\gamma \to 0} q(p,n,\gamma) = n,
\end{equation}
which leads us to the discussion of the delicate limiting case,
$\gamma =0$ in the minimization problem \eqref{eq.variational}. As
mentioned earlier in this Introduction, for homogeneous, $f\equiv
0$, linear, $p=2$, jets and cavities problems, Lipschitz regularity
estimates have been established in the one-phase and two-phase
case, respectively in \cite{AC} and \cite{ACF}. Heterogeneous,
two-phase versions of the problem could only be approached  quite
recently, with the aid of the \textit{almost} monotonicity
formula, \cite{CJK}. However, the Caffarelli-Jerison-Kenig
monotonicity formula requires a one-side bound for the
non-homogeneous term $f(X)$, namely, $f(X) \ge -C$. Thus, even for
linear problems, $p=2$, Lipschitz estimates for minimizers of
\eqref{eq.variational}, $\gamma = 0$, are only known if $f \in
L^\infty(\Omega)$. We further point out that the integrability
exponent obtained in \eqref{int of f gamma=0} is a borderline
condition, as it divides the regularity theory for (non-singular)
Poisson equations, $\text{L}u = f$, between continuity estimates
when $f \in L^{n-\epsilon}$ and differentiability properties when
$f \in L^{n+\epsilon}$. The optimal regularity theory for the
conformal case $f \in L^n$ is rather delicate. It has been
recently established by the third author, \cite{T}, that solutions
to nonlinear equations $F(X, D^2u) = f(X) \in L^n$ has a universal
Log-Lipschitz modulus of continuity, i.e.,
$$
    |u(X) - u(Y)|  \lesssim |X-Y| \cdot \log |X-Y|.
$$
Such regularity is optimal in the context of heterogeneous equations
with $L^n$ right-hand-sides.  After some heuristic inferences, it
becomes reasonable to inquire whether minimizers of problem
\eqref{eq.variational}, with $\gamma = 0$, also has a universal
Log-Lipschitz modulus of continuity.  The second main result we
establish in this paper states that indeed minimizers of
$\mathcal{J}_0$ with sources $f \in L^n$ also enjoy such an optimal
universal modulus of continuity.

\begin{theorem}[Log-Lipschitz  regularity for $\gamma =0$] \label{thm 2} Let $u$ be a
minimizer of the problem \eqref{eq.variational}, with $\gamma =0$
and $f \in L^n(\Omega)$. Then $u$ is Log-Lipschitz continuous and
for any $\Omega' \Subset \Omega$, there exists a constant $C$ that
depends only on, $\Omega'$, $n$, $p$,
$|\varphi\|_{L^\infty(\Omega)}$, $\|f\|_{L^n(\Omega)}$,
$\lambda_{+}$ and $\lambda_{-}$, such that
$$
    |u(X) - u(Y) | \le  C |X-Y| \log|X-Y|.
$$
\end{theorem}

\medskip

In particular, Theorem \ref{thm 2} assures that $u \in
C_\text{loc}^{0,\tau}(\Omega)$ for any $\tau < 1$. We further
mention that Theorem \ref{thm 2} is sharp due to the borderline
integrability condition on the source $f$. We leave open the key question on whether functional $\mathcal{J}_0$
has a locally Lipschitz minimizer, provided $f \in L^q(\Omega)$ for
$q>n$. We highlight that this question remained open even for the
linear case $p=2$, as no \textit{almost} monotonicity formula can
be established unless the source function $f$ is bounded. A critical
analysis on the machinery employed in the proof of Log-Lipschitz
estimates, Theorem \ref{thm 2}, reveals that it should not be
possible to access  the $C^{0,1}$ regularity theory for minima of $\mathcal{J}_0$  
through pure energy considerations, even if
the source $f\in L^\infty$. 

\medskip

We further mention that Theorem \ref{thm 1} and Theorem \ref{thm 2} can be
established, with minor modifications, to
further involved energy functionals of the type
$$
    \tilde{\mathcal{G}_\gamma} (v) = \int_\Omega G(X, \nabla v) + G_\gamma(v) + g(X,v) dX,
$$
where $G$ is a $p$-degenerate kernel with $C^1$ coefficients,
$|G_\gamma|  \lesssim F_\gamma$ and $|g(X,v)| \le
\tilde{g}(X)|v|^m$, where $0\le m < p$ and $\tilde{g} \in
L^{\tilde{q}}$, for  $\tilde{q} \geq \max\left\lbrace
\frac{p}{p-m}, n \right\rbrace$. We have chosen to present our
results in a simpler setting as to further emphasize the new ideas designed in this work.

\medskip

The paper is organized as follows. In Section \ref{s. pre} we
gather few tools that we shall use in the proofs of Theorem
\ref{thm 1} and Theorem \ref{thm 2}. In
Section \ref{section exist and bound} we comment on existence and
establish universal $L^\infty$ bounds for minima of problem
\eqref{eq.variational}. Section \ref{s. proof thm 1} is devoted to
the proof of Theorem \ref{thm 1} and in Section \ref{s. proof thm
2} we establish Log-Lip estimates for cavitation problems, proving
therefore Theorem \ref{thm 2}. Under the condition $f \in L^q$,
$q> n$, in Section \ref{S. nondeg} we show sharp linear growth and
strong nondegeneracy properties for solutions to the cavitation
problem $\gamma =0$. In Section \ref{S. stability} we investigate stability properties for the family of free problems $\mathcal{J}_\gamma$ in terms of the singular parameter $0 \le \gamma \le 1$. More precisely we show that
local minima of functional $\mathcal{J}_\gamma$ converges to a
local minima of the functional $\mathcal{J}_0$, as $\gamma \to 0$.  

\medskip

\medskip

\noindent \textbf{Acknowledgement}  RL research has been funded by Capes and CNPq. OSQ has been partially supported by
FAPESP/SP-Brazil. ET thanks support from CNPq-Brazil.

\section{Preliminaries and some known tools} \label{s. pre}

In this section we gather some preliminaries results that we will
systematically use along the article. Initially, as mentioned within the Introduction, clearly one
should not expect solutions to the minimization problem
\eqref{eq.variational} to be smoother than $p$-harmonic functions.
Therefore, the regularity theory for degenerate elliptic operators
is a first key ingredient in understanding sharp estimates for
minima of $\mathcal{J}_\gamma$.

There are several different strategies to establish the
$C^{1,\alpha_p}$ regularity theory for $p$-harmonic functions, see
for instance \cite{dibened-na1983}, \cite{evans-jde},
\cite{lewis-iumj}, \cite{tolk-jde}, \cite{uhlemb-acta} and
\cite{uralt-68}. We state such result for future references.

\begin{theorem}[$C^{1,\alpha}$ estimates for $p$-harmonic functions]\label{thm reg p-harm}
Let $h \in W^{1,p}(B_1)$ satisfy $\Delta_p h = 0$ in $B_1$ in the
distributional sense. Then, there exist constants $C>0$ and $0<
\alpha_p < 1$, both depending only on dimension and $p$, such that
$$
    \|h\|_{C^{1,\alpha_p}(B_{1/2})} \le C \|h\|_{L^p(B_1)}.
$$
\end{theorem}

A particularly interesting approach was suggested by Lieberman in
\cite{lieberman}, where the re\-gu\-la\-ri\-ty theory for
$p$-harmonic functions is accessed through the following leading
integral oscillation decay lemma:

\begin{lemma}[Lieberman, \cite{lieberman}, Lemma 5.1] \label{l 2.3.5} Let $h$ be a $p$-harmonic function in $B_R \subset \mathbb{R}^n$. Then, for some positive constant $0< \alpha_p < 1$, there holds
$$
    \int_{B_r} \left |\nabla h(X) - \left (\nabla h \right)_r \right |^p dX \le  C  \left (\dfrac{r}{R} \right )^{n+p\alpha_p}  \int_{B_R} \left |\nabla h(X) - \left (\nabla h \right)_R \right |^p dX,
$$
where $C = C (n,p) > 0$ is a positive constant.
\end{lemma}

In Lemma \ref{l 2.3.5} and throughout this article  we  use the classical average notation
$$
    (\psi)_r:=\intav{B_r(X_0)}  \psi dX := \frac{1}{|B_r(X_0)|}\int_{B_{r}(X_0)}  \psi dX.
$$

Local H\"older continuity for heterogeneous equations $\Delta_p
\xi = f$ can be delivered by means of Harnack inequality, which
will be another fundamental tool in our analysis.

\begin{theorem}[Harnack inequality, e.g. \cite{Serrin63}] \label{thm Harnack} Let
$\xi \in W^{1,p}(B_R)$, $\xi\ge 0$ a.e., satisfy $\Delta_p \xi = f$
in $B_R$ in the distributional sense, with $f \in L^{q}\left( B_{R}
\right)$ and $q>\frac{n}{p}$. Then, there exists a constant $C_r>0$
depending only on $n$, $q$, $p$ and $R-r$ such that
$$
    \sup_{B_{r}}\xi \leq C_r\left\lbrace \inf_{B_{r}} \xi  + \left ( r^{p-\frac{n}{q}}\Vert
    f \Vert_{L^{q}\left( B_{R}\right) } \right )^{\frac{1}{p-1}} \right\rbrace
$$
for all $0< r \leq R$.
\end{theorem}

In the sequel, let us discuss some further inequalities that will
be used in the proofs of our main results. The estimates presented
herein have elementary character and are mostly known. We include
them for completeness purposes and courtesy to the readers.

\begin{lemma} \label{mon lemma} Let $\psi \in W^{1,p}(B_1)$, with $2\le p$ and $h \in W_\psi^{1,p}(B_1)$ solution to $\Delta_p h =0$ in $B_1$. Then, for a constant $c = c(n,p) > 0$, there holds
$$
    \int_{B_1} \left [ |\nabla \psi|^p - |\nabla h|^p \right ] dX \ge c \int_{B_1} |\nabla (\psi - h)|^p dX.
$$
\end{lemma}
\begin{proof} For each $0\le \tau \le 1$, let $\phi_\tau$
denote the linear interpolation between $\psi$ and $h$, i.e.,
$\psi_\tau := \tau \psi + (1-\tau) h$. From Fundamental Theorem of
Calculus we have
\begin{equation}\label{lemma 2.2 - Eq01}
        \int_{B_1} \left ( |\nabla \psi|^p - |\nabla h|^p
        \right ) dX =\int_0^1 \dfrac{d}{d\tau} \left ( \int_{B_1} |\nabla \psi_\tau|^p  dX \right ) d\tau.
\end{equation}
Passing the derivative through and using the fact that $\div\left
( |\nabla h|^{p-2} \nabla h \right ) \cdot (\psi-h) = 0$ in $B_1$, we find
\begin{equation}\label{lemma 2.2 - Eq02}
    \begin{array}{lll}
        \displaystyle \int_0^1 \dfrac{d}{d\tau} \left ( \int_{B_1} | \nabla \psi_\tau|^p  dX \right ) d\tau &=& p\displaystyle \int_0^1 d\tau \int_{B_1} \left (   |\nabla \psi_\tau|^{p-2} \nabla \psi_\tau -  |\nabla h|^{p-2} \nabla h \right )  \cdot \nabla (\psi-h) dX \\
        &=& p\displaystyle \int_0^1 \dfrac{1}{\tau} d\tau \int_{B_1} \left (   |\nabla \psi_\tau|^{p-2} \nabla \psi_\tau -  |\nabla h|^{p-2} \nabla h \right ) \cdot \nabla (\psi_\tau-h) dX,
    \end{array}
\end{equation}
because $\psi_\tau - h = \tau (\psi-h)$. The Lemma now follows easily
from the well known classical monotonicity
\begin{equation}\label{lemma 2.2 - Eq03}
    \langle |\xi_1|^{p-2} \xi_1  - |\xi_2|^{p-2} \xi_2, \xi_1
    - \xi_2 \rangle > c(n,p)   |\xi_1 - \xi_2|^p,
\end{equation}
for any pair of vectors $\xi_1, \xi_2 \in \mathbb{R}^n$. In fact, combining \eqref{lemma 2.2 - Eq01}, \eqref{lemma 2.2 - Eq02} and \eqref{lemma 2.2 - Eq03} we reach
$$
    \displaystyle \int_{B_1} \left (\nabla \psi|^p - |\nabla h|^p
        \right )  \ge  \displaystyle c(n,p) \int_0^1 \tau^{-1} d\tau \int_{B_1} |\nabla (\psi_\tau-h)|^p   \\
         = \displaystyle c(n,p) \int_0^1 \tau^{p-1} d\tau \int_{B_1} |\nabla (\psi -h)|^p,
$$
and the Lemma follows.
\end{proof}
%%%
\begin{lemma}
\label{a 1}
Let $\gamma \in \left( 0, 1 \right)$. For any positive scalars $a > 0,~  b > 0$ there holds
\begin{eqnarray}
\left( a + b \right)^{\gamma} < \left( a^{\gamma} + b^{\gamma} \right).
\end{eqnarray}
\end{lemma}
\begin{proof}
In fact, just notice that, since $\gamma -1$ is negative, we have
\begin{eqnarray}\label{a 1 eq01}
t^{\gamma-1} > \left( t + a \right)^{\gamma-1}, \ \ \ \forall t \in (0, \infty).
\end{eqnarray}
Then, integrate \eqref{a 1 eq01} from $0$ to $b$ to obtain the
desired inequality.
\end{proof}
Next we prove two useful asymptotic inequalities.

\begin{lemma}\label{l.elem.asymp.ineq} Let $0 \le \mu < 1$ and suppose a real function $\phi$ verifies
$$
    \phi(r) \le A \left ( r^{e_1} \phi(r)^\mu + r^{e_2} \right ),
$$
for $r$ small enough. Then $\phi(r) = \mathrm{O} \left( r^{\min \left (e_2, \frac{e_1}{1-\mu} \right )}\right )$
as $r$ approaches zero.
\end{lemma}

\begin{proof}
In fact, if $\phi(r) \lesssim r^{e_1} \phi(r)^\mu + r^{e_2}$, then for $\beta := \min \left (e_2, \frac{e_1}{1-\mu} \right ),$ there holds
$$
    \begin{array}{lll}
        \dfrac{\phi(r)}{r^\beta} &\lesssim& r^{e_1 - \beta} \phi(r)^\alpha + r^{e_2 - \beta} \\
        &\lesssim& \left (\dfrac{\phi(r)}{r^{\frac{\beta - e_1}{\mu}}} \right )^\mu + 1 \\
        &\lesssim& \left (\dfrac{\phi(r)}{r^{\beta}} \right )^\mu + 1,
    \end{array}
$$
since $\frac{\beta - e_1}{\alpha} \le \beta$, The above readily implies $\phi(r) = \text{O}(r^\beta)$ as claimed.
\end{proof}
\begin{lemma}
\label{elem.ineq}
Let $\phi(s)$ be a non-negative and non-decreasing function. Suppose that
\begin{eqnarray}
\phi \left( r \right) \leq C_{1} \left[ \left( \dfrac{r}{R}\right)^{\alpha} + \mu \right] \phi \left( R \right) + C_{2}R^{\beta}
\end{eqnarray}
for all $r \leq R \leq R_{0}$, with $C_{1}, \alpha, \beta$ positive constants and $C_{2}, \mu$ non-negative constants, $\beta < \alpha$. Then, for any $\sigma < \beta$, there exists a constant $\mu_{0}=\mu_{0}\left( C_{1}, \alpha, \beta, \sigma \right)$ such that if $\mu < \mu_{0}$, then for all $r \leq R \leq R_{0}$ we have
\begin{eqnarray}
\label{a 1.1}
\phi \left( r \right) \leq C_{3} \left( \dfrac{r}{R}\right)^{\sigma}\left[ \phi \left( R \right) + C_{2}R^{\sigma} \right]
\end{eqnarray}
where $C_{3}=C_{3}\left( C_{1}, \sigma - \beta   \right)$ is a positive constant. In turn,
\begin{eqnarray}
\label{a 1.1.1}
    \phi(r) \le C_{4} r^{\sigma},
\end{eqnarray}
where $C_{4}=C_{4}(C_{2}, C_{3}, R_{0}, \phi, \sigma)$ is a positive constant.
\end{lemma}
\begin{proof}
It suffices to show the estimate for $\sigma = \beta$. For $0 < \theta < 1$ and $R \leq R_{0}$ we have
\begin{eqnarray}
\phi \left( \theta R \right) & \leq & C_{1} \left[ \left( \dfrac{\theta R}{R}\right)^{\alpha} + \mu \right] \phi \left( R \right) + C_{2}R^{\beta} \\ \nonumber
& = & \theta^{\alpha} C_{1} \left[ 1 + \mu \theta^{-\alpha} \right] \phi \left( R \right) + C_{2}R^{\beta}.
\end{eqnarray}
We choose $0 < \theta < 1$ such that $2C_{1}\theta^{\alpha} = \theta^{\delta}$ with $\beta < \delta < \alpha$. Now we take
$\mu_{0} > 0$ satisfying $\mu_{0} \theta^{-\alpha} < 1$. Thus we obtain for all $R \leq R_{0}$
\begin{eqnarray}
\phi \left( \theta R \right) & \leq & \theta^{\delta} \phi \left( R \right) + C_{2}R^{\beta}.
\end{eqnarray}
Inductively we get
\begin{eqnarray}
\phi \left( \theta^{k+1} R \right) & \leq & \theta^{\delta} \phi \left( \theta^{k}R \right) + C_{2}\theta^{k\beta}R^{\beta} \\ \nonumber
& \leq & \theta^{\left( k +1 \right)\delta} \phi \left( \theta^{k}R \right) + C_{2}\theta^{k\beta}R^{\beta}\sum_{i=1}^{k}\theta^{i\left( \delta - \beta \right)} \\ \nonumber
& \leq & C_{3} \theta^{\left( k +1 \right)\delta}\left[ \phi \left( R \right) + C_{2}R^{\beta} \right]
\end{eqnarray}
for all $k \in \mathbb{N}$. Hence, taking $k$ such that $\theta^{k+1} R \leq r \leq \theta^{k} R$ we obtain (\ref{a 1.1}).\\
Finally, we have
\begin{eqnarray}
\phi \left( r \right)& \leq & C_{3} \left( \dfrac{r}{R_{0}}\right)^{\sigma}\left[ \phi \left( R_{0} \right) + C_{2}R_{0}^{\sigma} \right] \\ \nonumber
& = &\left[ \frac{C_{3}}{R_{0}^{\sigma}}(\phi \left( R_{0} \right) + C_{2}R_{0}^{\sigma})\right] r^{\sigma}
\end{eqnarray}
which proves inequality (\ref{a 1.1.1}).

\end{proof}

%%%%%%%%%%%%%%
\section{Existence and $L^\infty$ bounds of minimizers} \label{section exist and bound}

In this section we establish existence and pointwise bounds for a
minimum of the functional $\mathcal{J}_\gamma$. The arguments
presented herein works indistinctly for the cases $0 < \gamma \le
1$ and $\gamma = 0$.

%%%%%%%%%%%
\begin{theorem}[Existence and $L^\infty$ bounds]\label{l.bound}
Let $\Omega\subset\mathbb R^n$ be a bounded domain, $f \in L^q(\Omega)$, $q \ge  n$,  $\varphi \in W^{1,p}(\Omega)\cap L^\infty(\Omega)$ and  $0 < \lambda_{+} \not = \lambda_{-} < \infty$ be fixed. For each $0\le \gamma \le 1$, there exists a minimizer $u_\gamma$ to the energy functional
$$
    \mathcal{J}_\gamma(v):=\int_\Omega\left (|\nabla  v|^p+F_\gamma(v) + f(X)\cdot v  \right )dX,
$$
over $W^{1,p}_0 + \varphi$, where
$F_\gamma(v):=\lambda_+(v^+)^\gamma+\lambda_-(v^-)^\gamma$ and by
convention, $F_0(v):=\lambda_+ \chi_{\{v > 0 \}}
+\lambda_{-}\chi_{\{v \le 0 \}}.$ Furthermore, $u_\gamma$ is
bounded. More precisely,
$$
    \|u_\gamma\|_{L^\infty(\Omega)} \le C(n,p, \lambda_{+}, \lambda_{-}, \|\varphi\|_{L^\infty(\partial \Omega)},
    \|f\|_{L^q(\Omega)}).
$$
\end{theorem}
%%%%%%%%%%%
\begin{proof}
Let us label
\begin{eqnarray*}
    I_{0}:= \min \left \{ \mathcal{J}_\gamma(v)  : v \in W^{1,p}_0 + \varphi \right \}.
\end{eqnarray*}
Initially we show that $I_0 > - \infty$. Indeed, for any $v \in
W^{1,p}_0 + \varphi$, by Poincar\'e inequality there exists a
positive constant $c=c(n, p,\Omega, \Vert f \Vert_{L^{q}})>0$ such
that
\begin{eqnarray}
\label{bounded of Lp norm}
c\Vert v \Vert^{p}_{L^{p}} - c\Vert \phi \Vert^{p}_{L^{p}} - \Vert \nabla \phi \Vert^{p}_{L^{p}} \leq  \Vert \nabla v \Vert^{p}_{L^{p}}.
\end{eqnarray}
By H\"older inequality, since
\begin{eqnarray}
q \geq n > \frac{p}{p-1},
\end{eqnarray}
we have
\begin{eqnarray*}
\label{}
\left | \int_{\Omega} f(X)v dX \right | \leq \Vert f \Vert_{L^{\frac{p}{p-1}}} \Vert v \Vert_{L^{p}}
\leq C_{1}(n,p,\Omega)\Vert f \Vert_{L^{q}} \Vert v \Vert_{L^{p}},
\end{eqnarray*}
which combined with (\ref{bounded of Lp norm}) gives
\begin{eqnarray*}
\label{}
- C - c\Vert \phi \Vert^{p}_{L^{p}} - \Vert \nabla \phi \Vert^{p}_{L^{p}} & \leq & \Vert \nabla v \Vert^{p}_{L^{p}} - C_{1}(n,p,\Omega)\Vert f \Vert_{L^{q}} \Vert v \Vert_{L^{p}}.
\end{eqnarray*}
Finally, we reach
\begin{eqnarray}
\label{th 2.1.a}
- C - c\Vert \phi \Vert^{p}_{L^{p}} - \Vert \nabla \phi \Vert^{p}_{L^{p}} \leq \Vert \nabla v \Vert^{p}_{L^{p}} - C_{1}(n,p,\Omega)\Vert f \Vert_{L^{q}} \Vert v \Vert_{L^{p}} \leq \mathcal{J}_\gamma(v).
\end{eqnarray}
\par
Let us now show existence of a minimum. Let $v_{j} \in
W^{1,p}_\phi(\Omega)$ be a minimizing sequence. For $j \gg 1$,
\begin{eqnarray*}
\mathcal{J}_\gamma (v_j)   & \leq & I_{0} + 1.
\end{eqnarray*}
From (\ref{th 2.1.a}) and  H\"older inequality we obtain
\begin{eqnarray}
\label{th 2.1.2} \int_{\Omega}\vert \nabla v_{j} \vert^{p} dX &
\leq & C \Vert v_{j} \Vert_{L^{p}} + I_{0} + 1.
\end{eqnarray}
By Poincar\'e inequality we estimate
\begin{eqnarray}
\label{th 2.1.3}
C \Vert v_{j} \Vert_{L^{p}} & \leq & C\left(\Vert \nabla v_{j} \Vert_{L^{p}} + \Vert \nabla \phi \Vert_{L^{p}} + \Vert \phi \Vert_{L^{p}} \right).
\end{eqnarray}
Also we have
\begin{eqnarray}
\label{th 2.1.4}
C\Vert \nabla v_{j} \Vert_{L^{p}} \leq C + \frac{1}{2} \Vert \nabla v_{j} \Vert^{p}_{L^{p}}.
\end{eqnarray}
Combining (\ref{th 2.1.2}), (\ref{th 2.1.3}) and (\ref{th 2.1.4})
we reach
\begin{eqnarray}
\label{bounded of Lp norm of gradient}
\int_{\Omega}\vert \nabla v_{j} \vert^{p} dX & \leq & C\left(\Vert \nabla \phi \Vert_{L^{p}} + \Vert \phi \Vert_{L^{p}} \right) + I_{0} + 1.
\end{eqnarray}
Thus, using  Poincar\'e inequality once more, we conclude that
$\{v_{j} - \phi\}$ is a bounded sequence in $W_0^{1,p}\left(
\Omega \right)$. By reflexivity, there is a function $u \in
W_{\phi}^{1,p}\left( \Omega \right)$ such that, up to a
subsequence,
\begin{eqnarray*}
v_{j} \rightarrow  u  \ \ \mbox{weakly} \ \mbox{in} \
W^{1,p}\left( \Omega \right), \quad v_{j} \rightarrow  u \ \
\mbox{in} \ L^{p}\left( \Omega \right), \quad v_{j} \rightarrow u
\ \ \text{a.  e.} \ \mbox{in} \ \Omega.
\end{eqnarray*}
From lower semicontinuity of norms, we readily obtain
\begin{eqnarray*}
\int_{\Omega} |\nabla u |^{p} dX & \leq & \liminf_{j \rightarrow \infty} \int_{\Omega}|\nabla v_{j} |^{p}dX.
\end{eqnarray*}
By pointwise convergence we have, in the case  $0 < \gamma \leq 1$,
\begin{eqnarray*}
\int_{\Omega} F_{\gamma}(u) + f(X)u dX &\leq & \liminf_{j \rightarrow \infty} \int_{\Omega}F_{\gamma}(v_{j}) + f(X)v_{j}dX.
\end{eqnarray*}
For $\gamma = 0$, recalling that we are working under the regime
$\lambda_{+} > \lambda_{-}$, we have,
\begin{eqnarray*}
\int_{\Omega} \lambda_{-}\chi_{\left\lbrace u \leq 0 \right\rbrace }dX & = & \int_{\left\lbrace u \leq 0 \right\rbrace} \lambda_{-}\chi_{\left\lbrace v_{j} > 0 \right\rbrace }dX + \int_{\left\lbrace u \leq 0 \right\rbrace} \lambda_{-}\chi_{\left\lbrace v_{j} \leq 0 \right\rbrace }dX \\ \nonumber
&\leq & \int_{\left\lbrace u \leq 0 \right\rbrace} \lambda_{+}\chi_{\left\lbrace v_{j} > 0 \right\rbrace }dX + \int_{\Omega} \lambda_{-}\chi_{\left\lbrace v_{j} \leq 0 \right\rbrace }dX.
\end{eqnarray*}
Thus,
\begin{eqnarray*}
\int_{\Omega} \lambda_{-}\chi_{\left\lbrace u \leq 0 \right\rbrace }dX&\leq & \liminf_{j \rightarrow \infty}\left( \int_{\left\lbrace u \leq 0 \right\rbrace} \lambda_{+}\chi_{\left\lbrace v_{j} > 0 \right\rbrace }dX + \int_{\Omega} \lambda_{-}\chi_{\left\lbrace v_{j} \leq 0 \right\rbrace }dX\right).
\end{eqnarray*}
On the other hand, since $v_{j} \rightarrow u  \ \ \text{a.  e.} \
\mbox{in} \ \Omega $, we have
\begin{eqnarray*}
\int_{\Omega} \lambda_{+}\chi_{\left\lbrace u > 0 \right\rbrace }dX & = & \int_{\left\lbrace u > 0 \right\rbrace} \lambda_{+}\left( \lim_{j \rightarrow \infty}\chi_{\left\lbrace v_{j} > 0 \right\rbrace }\right) dX \\ \nonumber
&=& \lim_{j \rightarrow \infty}\int_{\left\lbrace u > 0 \right\rbrace} \lambda_{+}\chi_{\left\lbrace v_{j} > 0 \right\rbrace }dX.
\end{eqnarray*}
Hence,
\begin{eqnarray*}
\int_{\Omega} F_{0}(u)dX \leq  \liminf_{j \rightarrow
\infty}\int_{\Omega} F_{0}(v_{j})dX .
\end{eqnarray*}
In conclusion,
\begin{eqnarray*}
\mathcal{J}_\gamma(u) &\leq & \liminf_{j \rightarrow \infty} \mathcal{J}_\gamma(v_j) = I_0,
\end{eqnarray*}
for $0 \le \gamma \le 1$, which proves the existence a minimizer.

\medskip

Let us now turn our attention to $L^\infty$ bounds of  $u_\gamma
$, which hereafter in  this proof we will only refer as $u$. Let
us label
$$
    j_{_{0}}  := \left \lceil \sup_{\partial \Omega}\phi  \right \rceil,
$$
that is, the smallest natural number above $\sup_{\partial \Omega}\phi$.    For each $j \geq j_{0}$ we define the truncated function $u_{j}: \Omega \rightarrow \mathbb{R}$ by
\begin{eqnarray}
\label{th 2.2.6}
u_{j}= \left \{
\begin{array}{ccl}
j\cdot \mbox{sing}(u)  &\mbox{if}& |u| > j \\
u  &\mbox{if}&  |u| \leq j.
\end{array}
\right.
\end{eqnarray}
where $\mbox{sing}(u) = 1$ if $u \geq 0$ and $\mbox{sing}(u) = -1$
else. If we denote $A_{j} := \left\lbrace |u| > j \right\rbrace$,
we have, for each $j > j_0$
\begin{eqnarray}
u = u_{j} \ \ \mbox{in} \ \ A^{c}_{j}\ \ \mbox{and} \ \ u_{j} = j
\cdot \mbox{sing}(u)  \ \ \mbox{in} \ \ A_{j}.
\end{eqnarray}
Thus, by minimality of $u$, there holds, for $0 < \gamma \leq 1$,
\begin{eqnarray}
\label{th 2.2.7}
\int_{A_{j}} \vert \nabla u \vert^{p} dX & = &  \int_{\Omega}  \vert \nabla u \vert^{p} - \vert \nabla u_{j} \vert^{p} dX \\ \nonumber
& \leq &  \int_{A_{j}}  f \left(  u_{j} - u \right) dX + \int_{A_{j}}  \lambda_{+} \left( (u_{j}^+)^\gamma - (u^+)^\gamma \right) dX \\ \nonumber
& + & \int_{A_{j}}  \lambda_{-} \left( (u_{j}^-)^\gamma - (u^-)^\gamma \right) dX.
\end{eqnarray}
Notice that
\begin{eqnarray}
\label{}
\int_{A_{j}}  f \left( u_{j} - u \right) dX &=& \nonumber \int_{A_{j} \cap \left\lbrace u > 0\right\rbrace }  f \left(  j - u \right) dX + \int_{A_{j} \cap \left\lbrace u \leq 0\right\rbrace }  f \left( u - j \right) dX \\ \nonumber
& \leq & 2\int_{A_{j}} |f| \left( |u| - j \right) dX.
\end{eqnarray}
Moreover, we have
\begin{eqnarray}
\label{}
\lambda_{+} \int_{A_{j}}\left( (u_{j}^+)^\gamma - (u^+)^\gamma \right) dX & = & \nonumber \lambda_{+}\int_{A_{j} \cap \left\lbrace u > 0 \right\rbrace}\left( j^\gamma - |u|^\gamma \right) dX \\ \nonumber
& + & \lambda_{+}\int_{A_{j} \cap \left\lbrace u \leq 0\right\rbrace}\left( (-j)^{+})^\gamma - (u^{+})^\gamma \right) dX \\ \nonumber
& \leq & 0
\end{eqnarray}
and
\begin{eqnarray}
\label{}
\lambda_{-} \int_{A_{j}}\left( (u_{j}^-)^\gamma - (u^-)^\gamma \right) dX & = & \nonumber \lambda_{-}\int_{A_{j} \cap \left\lbrace u > 0 \right\rbrace}\left( (j)^{-})^\gamma - (u^{-})^\gamma \right) dX \\ \nonumber
& + & \lambda_{-}\int_{A_{j} \cap \left\lbrace u \leq 0\right\rbrace}\left( j^\gamma - |u|^\gamma \right) dX \\ \nonumber
& \leq & 0.
\end{eqnarray}
Then, we find
\begin{eqnarray}
\label{}
\int_{A_{j}} F_{\gamma}(u_{j}) - F_{\gamma}(u)\leq  0.
\end{eqnarray}
For $\gamma = 0$ it suffices to notice that $u_{j} > 0$ and $u$
have the same sign. From the range of truncation we consider, it
follows  that $\left( |u| - j \right)^{+} \in W_{0}^{1,p}\left(
\Omega \right)$. Hence, applying H\"older inequality and
Gagliardo-Nirenberg inequality, we find
\begin{eqnarray}
\label{th 2.2.9}
\int_{A_{j}} \vert f \vert  \left( |u| - j \right)^{+}dX & \leq & \nonumber \Vert f \Vert_{L^{\frac{p}{p - 1}}} \Vert  \left( |u| - j \right)^{+} \Vert_{L^{p}\left( A_{j}\right)} \\ \nonumber
& \leq & \Vert f \Vert_{L^{q}} | A_{j}|^{1 - \frac{1}{p^{*}} - \frac{1}{q}} \Vert  \nabla u \Vert_{L^{p}\left( A_{j}\right)},
\end{eqnarray}
where $p^{*}:= \frac{np}{n-p}$. Young inequality gives,
\begin{eqnarray}
\label{th 2.2.10}
\Vert f \Vert_{L^{q}} | A_{j}|^{1- \frac{1}{p*} - \frac{1}{q}} \Vert \nabla u \Vert_{L^{p}\left( A_{j}\right)} \leq C | A_{j} |^{\frac{p}{p-1} - \frac{p}{q(p-1)} - \frac{p}{p*\left( p-1\right)}} + \frac{1}{2}\Vert \nabla u \Vert_{L^{p}\left( A_{j} \right)}^{p} .
\end{eqnarray}
Combining (\ref{th 2.2.7}) and (\ref{th 2.2.10}) we obtain
\begin{eqnarray}
\int_{A_{j}} \vert \nabla u \vert^{p} dX \leq C | A_{j}|^{1 - \frac{p}{n} + \varepsilon},
\end{eqnarray}
where $\varepsilon = \frac{p(pq - n)}{nq\left( p - 1 \right) }$ and (see (\ref{bounded of Lp norm}) and (\ref{bounded of Lp norm of gradient}) substituting $I_{0}$ by $\mathcal{J}_\gamma(\varphi)$)
\begin{eqnarray}
\Vert u \Vert_{L^{1}\left( A_{j_{0}} \right)} \leq | A_{j_{0}} |^{\frac{p-1}{p}}\Vert u \Vert_{L^{p}\left( A_{j_{0}} \right)} \leq C.
\end{eqnarray}
Boundedness of $u$ now follows from a general machinery, see for instance,  \cite{ON}, Chap. 2, Lemma 5.2, Page 71.
\end{proof}

\begin{remark}
A  consequence of $L^\infty$ estimates for a minimum $u$ to the
functional $\mathcal{J}_{\gamma}$ is the universal control of $u$
in $W^{1,p}$. In fact, we have
\begin{eqnarray}
\int_{\Omega} |\nabla u|^{p} dx & \leq &  \mathcal{J}_\gamma(\varphi) - \int_{\Omega} F_{\gamma}(u)dX + \int_{\Omega} |f(X)||u|dX \\ \nonumber
&\leq &\mathcal{J}_\gamma(\varphi) + C(n,p,\Omega,\Vert f \Vert_{L^{q}})\\ \nonumber
&\leq &  C,
\end{eqnarray}
where $C=C(n,p,\Omega,\varphi,\Vert f \Vert_{L^{q}})>0$ is a
positive constant. Here we used the elementary inequality
$t^{\gamma} \leq \max\left\lbrace 1, t\right\rbrace$, for $t>0$
and $0\le \gamma \le 1$. In conclusion,
\begin{eqnarray}
\label{boundedness W^{1,p}}
\Vert u \Vert_{W^{1,p}} \leq C.
\end{eqnarray}

\end{remark}

\medskip

We close up this Section by stating the Euler-Lagrange equation
associated to the functional $\mathcal{J}_\gamma$, $0\le \gamma
\le 1$ as well as the flux balance -- also known as the free
boundary condition -- satisfied by a minimum $u_0$ to
$\mathcal{J}_0$, through the free boundary. The proofs of this
facts are rather standard and we omit them here.

\begin{proposition}\label{PDE} Let $u_\gamma$ be a minimum to the functional $\mathcal{J}_\gamma$, $0\le \gamma \le 1$. Then $u_\gamma$ solves
\begin{equation} \label{PDE for min}
    \Delta_p u=\frac{\gamma}{p}\left ( \lambda_+(u^+)^{\gamma-1}\chi_{\{u>0\}}-
    \lambda_-(u^-)^{\gamma-1}\chi_{\{u\le0\}}\right ) + \frac{1}{p} f(X) \quad \text{in}
    \quad\Omega,
\end{equation}
in the distributional sense. Also, if $u_0$ is a minimum of
$\mathcal{J}_0$, with $|\{u_0 = 0 \}| = 0$, $f \in L^q\left(
\Omega \right)$, $q >n$, $X_0 \in \mathfrak{F}^{+}(u_0) \cup
\mathfrak{F}^{-}u_0)$ a generic free boundary point and $B$ a ball
centered at $X_0$. Then for any $\Phi \in C^1_0(B, \mathbb{R}^n)$,
there holds
$$
    \begin{array}{l}
   \displaystyle  \lim\limits_{\epsilon_1 \searrow 0} \int\limits_{B \cap \{u_0 = \epsilon_1\}} \langle\left ((p-1)|\nabla u_0|^p - \lambda_{+} \right ) \nu_1, \Phi \rangle d\mathcal{H}^{n-1}   \\
    +  \\
    \displaystyle  \lim\limits_{\epsilon_2 \nearrow 0} \int\limits_{B \cap \{u_0 = \epsilon_2\}} \langle \left ((p-1)|\nabla u_0|^p - \lambda_{-} \right ) \nu_2 , \Phi \rangle d\mathcal{H}^{n-1} \\
    = 0,
    \end{array}
$$
where $\nu_1$ and $\nu_2$ denote the outward normal vector on $B
\cap \{u_0 = \epsilon_1\}$ and $B \cap \{u_0 = \epsilon_2\}$
respectively. In particular, the flux balance
$$
    |\nabla u_0^{+}|^p - |\nabla u_0^{-}|^p = \dfrac{1}{p-1} \left (\lambda_{+} - \lambda_{-} \right ),
$$
holds along any $C^{1,\alpha}$ piece of the free boundary.
\end{proposition}

\section{Sharp $C^{1,\alpha}$ estimates for minima} \label{s. proof thm 1}

This Section is devoted to the proof of Theorem \ref{thm 1}, which
assures optimal H\"older continuity estimates for the gradient of
minima of the energy functional $\mathcal{J}_\gamma$, for $0< \gamma
\le 1$ and $q > n$. The borderline situation $\gamma = 0$ and $f \in
L^n$ will be addressed in the next Section.

Hereafter in this Section, $u = u_\gamma$ denotes a minimizer of the
functional $\mathcal{J}_\gamma$, with $0 < \gamma \le 1$. Theorem
\ref{thm 1} concerns an optimal interior regularity result;
therefore, in order to prove such interior estimate, we fix an
arbitrary point $X_0\in\Omega$ and $R>0$ such that
$R<\dist(X_0,\partial\Omega)$. We will show that $u \in
C^{1,\alpha}$ at $X_0$, for $\alpha$ as in  \eqref{sharp reg intro}.

In the sequel we show the first main step in our strategy to obtain sharp regularity theory for minima of the energy $\mathcal{J}_\gamma$.

\begin{lemma}[Comparison with $p$-harmonic functions] \label{key osc lemma p-harm} Let $u \in W^{1,p}(B_R)$ and $h  \in W^{1,p}(B_R)$ satisfy $\Delta_p h = 0$ in $B_R$ in the distributional sense. Then,  there exists a positive constant $C = C(n,p) > 0$ depending  on dimension and $p$  such that for each $0 < r \le R$, there holds
\begin{eqnarray}
 \int_{B_r} \left |\nabla u(X) - \left (\nabla u \right)_r \right |^p dX & \le & C  \left (\dfrac{r}{R} \right )^{n+p\alpha_p}  \int_{B_R} \left |\nabla u(X) - \left (\nabla u \right)_R \right |^p dX \\ \nonumber
& + & C\int_{B_R} |\nabla u(X) - \nabla h(X)|^p dX ,
\end{eqnarray}
where $0< \alpha_p < 1$ is the optimal exponent in Lemma \ref{l 2.3.5}, which, in turn, reveals the sharp $C^{1,\alpha}$ estimate stated in Theorem \ref{thm reg p-harm}.
\end{lemma}

\begin{proof}
For each $r \in \left( 0, R \right]$ we estimate,
\begin{eqnarray}
\label{l 2.3.1}
\int_{B_r} \left |\nabla u(X) - \left( \nabla u \right)_r \right |^p dX  & \leq &  C_p \int_{B_r} \left |\nabla u(X) - \left( \nabla h \right)_{r} \right |^{p} dX \\ \nonumber
& + & C_p \int_{B_r} \left |\left( \nabla u \right)_r  - \left( \nabla h \right)_{r} \right|^{p} dX ,
\end{eqnarray}
for a constant $C_p$ that depends only on $p$. Analogously, we obtain
\begin{eqnarray}
\label{l 2.3.2}
\int_{B_r} \left |\nabla u(X) - \left( \nabla h \right)_r \right |^p dX  & \leq & C_p  \int_{B_r} \left |\nabla u(X) - \nabla h(X) \right |^{p} dX \\ \nonumber
& + & C_{p}\int_{B_r} \left |\nabla h(X) - \left( \nabla h\right)_{r} \right |^{p} dX .
\end{eqnarray}
In the sequel, we apply  H\"older inequality and estimate
\begin{eqnarray}
\label{l 2.3.3}
\int_{B_r} \left |\left( \nabla u \right)_r  - \left( \nabla h\right)_{r} \right |^{p} dX &=& \dfrac{1}{\left |B_{r} \right |^{p-1}} \left | \int_{B_r} \left ( \nabla u(X) - \nabla h(X) \right ) dX \right |^{p} \\ \nonumber
&\leq & \dfrac{1}{\left |B_{r} \right |^{p-1}}\left( \int_{B_r} \left |\nabla u(X) - \nabla h(X) \right | dX\right)^{p} \\ \nonumber
&\leq & \dfrac{1}{\left |B_{r} \right |^{p-1}}\left\lbrace \left |B_{r} \right |^{1-\frac{1}{p}} \left( \int_{B_r} \left |\nabla u(X) - \nabla h(X) \right |^{p} dX\right)^{\frac{1}{p}} \right\rbrace^{p} \\ \nonumber
&=& \int_{B_r} \left |\nabla u(X) - \nabla h(X) \right |^{p} dX.
\end{eqnarray}
Combining (\ref{l 2.3.1}), (\ref{l 2.3.2}) and (\ref{l 2.3.3})  we obtain
\begin{eqnarray}
\label{l 2.3.4}
\int_{B_r} \left |\nabla u(X) - \left (\nabla u \right)_r \right |^p dX  & \leq & C_p \int_{B_r} \left |\nabla h(X) - \left( \nabla h\right)_{r} \right |^{p} dX \\ \nonumber
& + & C_{p}\int_{B_r} \left |\nabla u(X) - \nabla h(X) \right |^{p} dX .
\end{eqnarray}
Interplaying the roles of $u$ and $h$ in (\ref{l 2.3.4}) and arguing in the bigger ball $B_R$, we find
\begin{eqnarray}
\label{l 2.3.7}
\int_{B_R} \left |\nabla h(X) - \left (\nabla h \right)_R \right |^p dX  & \leq &  C_p \int_{B_R} \left |\nabla u(X) - \left( \nabla u\right)_{R} \right |^{p} dX \\ \nonumber
& + & C_{p}\int_{B_R} \left |\nabla u(X) - \nabla h(X) \right |^{p} dX .
\end{eqnarray}
Now, in view of Lemma \ref{l 2.3.5} and \eqref{l 2.3.4} we can further estimate
\begin{equation}
\label{l 2.3.6}
    \begin{array}{lll}
        \displaystyle \int_{B_r} \left |\nabla u(X) - \left (\nabla u \right)_r \right |^p dX  &\leq&  \displaystyle  C(n,p)  \left (\dfrac{r}{R} \right )^{n+p\alpha_p}\int_{B_R} \left | \nabla h(X) - \left( \nabla h\right)_{R} \right |^{p} dX  \\
        &+&\displaystyle   C(n,p) \int_{B_R} \left |\nabla u(X) - \nabla h(X) \right |^{p} dX.
    \end{array}
\end{equation}
Hence, combining (\ref{l 2.3.7}) and (\ref{l 2.3.6}) we readily obtain
\begin{eqnarray}
\label{l 2.3.8}
\int_{B_r} \left |\nabla u(X) - \left (\nabla u \right)_r \right |^p dX  & \leq &  C(n,p) \left( \dfrac{r}{R} \right )^{n+p\alpha_p}\int_{B_R} \left |\nabla u(X) - \left( \nabla u\right)_{R} \right |^{p} dX \\ \nonumber
& + & C(n,p) \left[ 1 + \left( \dfrac{r}{R} \right )^{n+p\alpha_p} \right] \int_{B_R} \left |\nabla u(X) - \nabla h(X) \right |^{p} dX,
\end{eqnarray}
which finally implies
\begin{eqnarray}
\label{l 2.3.9}
\int_{B_R} \left |\nabla u(X) - \left (\nabla u \right)_r \right |^p dX & \le &  C \left (\dfrac{r}{R} \right )^{n+p\alpha_p}  \int_{B_R} \left |\nabla u(X) - \left (\nabla u \right)_R \right |^p dX \\ \nonumber
& + & C\int_{B_R} |\nabla u(X) - \nabla h(X)|^p dX ,
\end{eqnarray}
and the proof of Lemma \ref{key osc lemma p-harm} is concluded.
\end{proof}

\bigskip
We have now gathered all the tools and ingredients we need to
establish local H\"older continuity of the gradient of a minimum
of the energy functional $\mathcal{J}_\gamma$, $0< \gamma \le 1$.

\bigskip

\noindent \textit{Proof of Theorem \ref{thm 1}.} We start off the
proof by denoting, for writing convenience, $B_R:= B_R(X_0)$ and
$u = u_\gamma$ a given minimum of the functional
$\mathcal{J}_\gamma$, $0 < \gamma \le 1$. Let $h$ be the
$p$-harmonic function in $B_R$ that agrees with $u$ on the
boundary, i.e.,
$$
\Delta_p h = 0 \text{ in } B_R \quad \text{and} \quad h-u\in W^{1,p}_{0}(B_R).
$$
By Lemma \ref{key osc lemma p-harm} we have
\begin{eqnarray}
\label{t 2.6.1}
\int_{B_r} \left |\nabla u(X) - \left (\nabla u \right)_r \right |^p dX & \le & C  \left (\dfrac{r}{R} \right )^{n+p\alpha_{p}}  \int_{B_R} \left |\nabla u(X) - \left (\nabla u \right)_R \right |^p dX \\ \nonumber
& + & C\int_{B_R} |\nabla u(X) - \nabla h(X)|^p dX .
\end{eqnarray}
On the other hand, by the minimality of $u$
we have
\begin{equation}\label{ineq.minimum}
\int_{B_R}\big(|\nabla u|^p-|\nabla h|^p\big)dX\leq
\int_{B_R}\big(F_\gamma(h)-F_\gamma(u)\big)dX + \int_{B_R} f(X)(h - u)dX.
\end{equation}
Invoking Lemma \ref{mon lemma}, there exists a constant
$C_3=C_3(p,n)>0$ such that
\begin{equation}\label{ineq.p2}
C_3\int_{B_R}\big(|\nabla u|^p-|\nabla h|^p\big)dX\geq
\int_{B_R}|\nabla(u-h)|^pdX.
\end{equation}
Moreover, we have
\begin{eqnarray*}
\int_{B_R}F_\gamma(h)-F_\gamma(u)dX = \lambda_{+}\int_{B_R} \left[ \left( h^{+} \right)^{\gamma} - \left( u^{+} \right)^{\gamma} \right]dX + \lambda_{-}\int_{B_R} \left[ \left( h^{-} \right)^{\gamma} - \left( u^{-} \right)^{\gamma} \right]dX
\end{eqnarray*}
with
\begin{eqnarray}
\int_{B_R} \left[ \left( h^{+} \right)^{\gamma} - \left( u^{+} \right)^{\gamma} \right]dX & \leq & \int_{\left\lbrace h^{+} > u^{+} \right\rbrace}  \left[ \left( h^{+} \right)^{\gamma} - \left( u^{+} \right)^{\gamma} \right]dX \\ \nonumber
& = & \int_{\left\lbrace h^{+} > u^{+} \right\rbrace \cap \left\lbrace  u^{+} > 0 \right\rbrace}  \left[ \left( h^{+} \right)^{\gamma} - \left( u^{+} \right)^{\gamma} \right]dX \\ \nonumber
& + & \int_{\left\lbrace h^{+} > u^{+} \right\rbrace \cap \left\lbrace u^{+} = 0 \right\rbrace} \left( h^{+} - u^{+} \right)^{\gamma}dX.
\end{eqnarray}
Notice furthermore that
\begin{eqnarray}
 \int_{\left\lbrace h^{+} > u^{+} \right\rbrace \cap \left\lbrace u^{+} = 0 \right\rbrace} \left( h^{+} - u^{+} \right)^{\gamma}dX \leq \int_{\left\lbrace h^{+} > u^{+} \right\rbrace \cap \left\lbrace u^{+} = 0 \right\rbrace} \left( h - u \right)^{\gamma}dX.
\end{eqnarray}
By Lemma \ref{a 1}  there holds
\begin{eqnarray}
\int_{\left\lbrace h^{+} > u^{+} \right\rbrace \cap \left\lbrace u^{+} > 0 \right\rbrace}  \left[ \left( h^{+} \right)^{\gamma} - \left( u^{+} \right)^{\gamma} \right]dX & \leq & \int_{\left\lbrace h^{+} > u^{+} \right\rbrace \cap \left\lbrace u^{+} > 0 \right\rbrace}\left( h^{+} - u^{+} \right)^{\gamma}dX \\ \nonumber
& = & \int_{\left\lbrace h^{+} > u^{+} \right\rbrace \cap \left\lbrace  u^{+} > 0 \right\rbrace}\left( h - u \right)^{\gamma}dX \\ \nonumber
& \leq & \int_{B_R}| h - u |^{\gamma} dX.
\end{eqnarray}
Analogously, we obtain
\begin{eqnarray}
\int_{B_R} \left[ \left( h^{-} \right)^{\gamma} - \left( u^{-}
\right)^{\gamma} \right]dX  & \leq  & \int_{\left\lbrace h^{-} >
u^{-} \right\rbrace \cap \left\lbrace  u^{-} > 0 \right\rbrace}
\left[ \left( h^{-} \right)^{\gamma} - \left(
u^{-}\right)^{\gamma} \right]dX \\ \nonumber & + &
\int_{\left\lbrace h^{-} > u^{-} \right\rbrace \cap \left\lbrace
u^{-} = 0 \right\rbrace} \left( u - h \right)^{\gamma}dX,
\end{eqnarray}
with
\begin{eqnarray}
\int_{\left\lbrace h^{-} > u^{-} \right\rbrace \cap \left\lbrace u^{-} > 0 \right\rbrace} \left[ \left( h^{-} \right)^{\gamma} - \left( u^{-}\right)^{\gamma} \right]dX & \leq & \int_{\left\lbrace h^{-} > u^{-} \right\rbrace \cap \left\lbrace  u^{-} > 0 \right\rbrace}\left( h^{-} - u^{-}\right)^{\gamma}dX \\ \nonumber
& = & \int_{\left\lbrace h^{-} > u^{-} \right\rbrace \cap \left\lbrace u^{-} > 0 \right\rbrace}\left(  u - h \right)^{\gamma}dX \\ \nonumber
& \leq & \int_{B_R}|h - u|^{\gamma}dX .
\end{eqnarray}
Hence, we find
\begin{eqnarray}
\label{t 2.6.2}
\int_{B_R}F_\gamma(h)-F_\gamma(u)dX \leq C\int_{B_R}|h - u|^{\gamma}dX ,
\end{eqnarray}
where $C=C\left( \lambda_{+}, \lambda_{-} \right)$ is a positive constant.

\medskip

Combining \eqref{ineq.p2}, \eqref{ineq.minimum} and employing
H\"older inequality followed by Poincar\'e inequality and (\ref{t 2.6.2}) we obtain
\begin{equation}\label{ineq.pb2}
\begin{split}
\int_{B_R}|\nabla(u-h)|^pdX&\leq C_3\int_{B_R}F_\gamma(h)-F_\gamma(u)dX\\
&\leq C_4\int_{B_R}|u-h|^\gamma dX\\ \nonumber
&\leq
C_5\Big(\int_{B_R}|\nabla(u-h)|^pdX\Big)^{\gamma/p}|B_R|^{1+\gamma/n-\gamma/p},
\end{split}
\end{equation}
where $C_4$ and $C_5$ depend on $p$, $n$, $\lambda_+$ and
$\lambda_-$. Thus, by Young inequality we reach the following estimate
\begin{eqnarray}\label{ineq.p3}
\int_{B_R}F_\gamma(h)-F_\gamma(u)dX & \leq & \nonumber C\left( p, \gamma \right) \big[C(p,n,\lambda_+,\lambda_-)\big]^{p/(p-\gamma)}
|B_R|^{1+1/n(p\gamma/(p-\gamma))} \\ \nonumber
& + & \dfrac{1}{4}\Vert \nabla(u-h) \Vert^{p}_{L^{p}} \\
& \leq & C(p)\big[C(p,n,\lambda_+,\lambda_-)\big]^{p/(p-1)}
|B_R|^{1+1/n(p\gamma/(p-\gamma))} \\ \nonumber
& + & \dfrac{1}{4}\Vert \nabla(u-h) \Vert^{p}_{L^{p}},
\end{eqnarray}
where $C\left( p, \gamma \right)=\left(
\frac{4\gamma}{p}\right)^{\frac{\gamma}{p-\gamma}}\left( \frac{p -
\gamma}{p}\right)$ and $C\left( p \right)=\left(
\frac{4}{p}\right)^{\frac{1}{p-1}}$. H\"older inequality and
Poincar\'e inequality yield
\begin{eqnarray}
\label{ineq.p3.1}
\int_{B_R} f(X)(h - u)dX & \leq & \Vert f \Vert_{L^{q}}|B_R|^{\frac{p-1}{p} - \frac{1}{q}}\Vert u-h \Vert_{L^{p}} \\ \nonumber
& \leq & \Vert f \Vert_{L^{q}}|B_R|^{\frac{p-1}{p} - \frac{1}{q} + \frac{1}{n}}\Vert \nabla(u-h) \Vert_{L^{p}}.
\end{eqnarray}
Thus, applying  Young inequality once more, we reach
\begin{eqnarray}
\label{ineq.p3.2} \int_{B_R} f(X)(h - u)& \leq & C(p)(\Vert f
\Vert_{L^{q}})^{\frac{p}{p-1}}|B_R|^{\frac{p}{p-1}(\frac{p-1}{p} -
\frac{1}{q} + \frac{1}{n})}\Vert \nabla(u-h) \Vert^{p}_{L^{p}}\\
\nonumber & + &\dfrac{1}{4}\Vert \nabla(u-h) \Vert_{L^{p}} \\
\nonumber & = &  C(p)(\Vert f
\Vert_{L^{q}})^{\frac{p}{p-1}}|B_R|^{1 + \frac{1}{n}\left[
\frac{(q-n)p}{(p-1)q}\right]} + \dfrac{1}{4}\Vert \nabla(u-h)
\Vert_{L^{p}}.
\end{eqnarray}
Replacing \eqref{ineq.p3}  and \eqref{ineq.p3.2} in \eqref{t 2.6.1} we easily obtain
\begin{eqnarray}\label{ineq.endlemma}
 \int_{B_r}|\nabla u-(\nabla u)_r|^pdX & \leq & \nonumber C(n,p, \alpha_{p})\left (\dfrac{r}{R} \right )^{n+p\alpha_{p}}  \int_{B_R} \left |\nabla u(X) - \left (\nabla u \right)_R \right |^p dX \\ \nonumber
& + & C(n,p, \alpha_{p})C(n,p,\lambda_+,\lambda_-)
|B_R|^{1+1/n(p\gamma/(p-\gamma))} \\ \nonumber
& + & C(n,p, \alpha_{p})(\Vert f \Vert_{L^{q}})^{\frac{p}{p-1}}|B_R|^{1 + \frac{1}{n}\left[ \frac{(q-n)p}{(p-1)q}\right]} \\ \nonumber
& \leq & C\left (\dfrac{r}{R} \right )^{n+p\alpha_{p}}  \int_{B_R} \left |\nabla u(X) - \left (\nabla u \right)_R \right |^p dX \\ \nonumber
& + & C R^{n+p\gamma/(p-\gamma)} + CR^{n+p\frac{(q-n)}{(p-1)q}}.
\end{eqnarray}
where $C=C(n,p,\lambda_{+}, \lambda_{-},\alpha_{p}, \Vert f \Vert_{L^{q}})$ is a positive constant. In view of Lemma \ref{elem.ineq} and $W^{1,p}$ bounds of $u$ we conclude
\begin{equation}\label{eq end thm 1}
 \displaystyle\intav{B_r(X_0)}|\nabla u-(\nabla u)_r|^pdX \leq  C\left ( n,p,\lambda_{+}, \lambda_{-}, \|f\|_{L^q(\Omega)}, \text{dist}(X_0, \partial \Omega) \right )  \cdot r^\alpha,
\end{equation}
for $\alpha$ entitled in \eqref{sharp reg intro}.  Finally Campanato's embedding Theorem (see for instance \cite{maly-ziemer-book}) gives the desired H\"older continuity of the gradient of $u$. The proof of Theorem \ref{thm 1} is complete. \qed

\bigskip

\begin{remark} \label{rmk sct4} It is important to notice that the estimates from Campanato's embedding Theorem are not uniform as $\gamma $ goes to zero. In fact, an inspection of the proof of such Theorem (see for instance \cite{maly-ziemer-book} Theorem 1.54) reveals that estimate \eqref{eq end thm 1} implies 
$$
	|\nabla u(X) - \nabla u(Y)| \le \dfrac{2^n \cdot C\left ( n,p,\lambda_{+}, \lambda_{-}, \|f\|_{L^q(\Omega)}, \text{dist}(X_0, \partial \Omega) \right ) }{2^\alpha - 1} |X-Y|^\alpha.
$$
 This is the reason why the constant in Theorem \ref{thm 1} do depend upon $\gamma$, even though the universal constant appearing in \eqref{eq end thm 1} does not depend upon $\gamma$. 
\end{remark}

%%%%%%%%%%%%%%%%%%%%%%

\section{Log-Lipschitz estimates} \label{s. proof thm 2}

In this Section we address sharp regularity for jets and cavities
type problems, i.e., $\gamma = 0$, with sources in the conformal
threshold case $f \in L^n(\Omega)$, where $n$ is the dimension of
the ambient. Hereafter $u = u_0$ denotes a minimizer of the energy
functional
\begin{equation}\label{cav functional}
    \mathcal{J}_0(v) := \int_\Omega\left (|\nabla
    v|^p+ \lambda_+ \chi_{\{v > 0 \}} +\lambda_{-}\chi_{\{v \le 0 \}} + f(X)\cdot  v \right )dX,
\end{equation}
for scalars  $0 \le \lambda_{-} < \lambda_{+} < \infty$. Existence
and pointwise bounds for $u_0$ is has been assured by Theorem
\ref{l.bound}.

\bigskip

\textit{Proof of Theorem \ref{thm 2}.} We start off by fixing an
arbitrary point $X_0\in\Omega$ and $R>0$ such that
$R<\dist(X_0,\partial\Omega)$. As before, we denote
$B_R:=B_R(X_0)$. We follow the initial steps of the proof of
Theorem \ref{thm 1}. Let $h$ be the $p$-harmonic function in $B_R$
that agrees with $u$ on the boundary, i.e.,
$$
\Delta_p h = 0 \text{ in } B_R \quad \text{and} \quad h-u\in W^{1,p}_{0}(B_R).
$$
By Lemma \ref{key osc lemma p-harm} we have
\begin{eqnarray}
\label{t 5.1}
\int_{B_r} \left |\nabla u(X) - \left (\nabla u \right)_r \right |^p dX & \le & C  \left (\dfrac{r}{R} \right )^{n+p\alpha_{p}}  \int_{B_R} \left |\nabla u(X) - \left (\nabla u \right)_R \right |^p dX \\ \nonumber
& + & C \int_{B_R} |\nabla u(X) - \nabla h(X)|^p dX .
\end{eqnarray}
On the other hand, by the minimality of $u$
we have
\begin{equation}\label{ineq.minimum 1}
\int_{B_R}\big(|\nabla u|^p-|\nabla h|^p\big)dX\leq
\int_{B_R}\big(F_0(h)-F_0(u)\big)dX + \int_{B_R} f(X)(h - u)dX.
\end{equation}
Readily one verifies that
\begin{eqnarray}
\label{ineq.p5.1}
\int_{B_R}\big(F_0(h)-F_0(u)\big)dX \leq C(\lambda_{+}, \lambda_{-})|B_R|.
\end{eqnarray}
As before, applying H\"older inequality and afterwards Poincar\'e inequality we obtain
\begin{eqnarray}
\label{ineq.p5.2}
\int_{B_R} f(X)(h - u)dX& \leq & \Vert f \Vert_{L^{n}}|B_R|^{\frac{p-1}{p} - \frac{1}{n}}\Vert u-h \Vert_{L^{p}} \\ \nonumber
& \leq & \Vert f \Vert_{L^{n}}|B_R|^{\frac{p-1}{p} - \frac{1}{n} + \frac{1}{n}}\Vert \nabla(u-h) \Vert_{L^{p}}.
\end{eqnarray}
Therefore, with the aid of Young inequality we estimate
\begin{eqnarray}
\label{ineq.p5.3}
\int_{B_R} f(X)(h - u) dX & \leq & \nonumber C(p)(\Vert f \Vert_{L^{n}})^{\frac{p}{p-1}}|B_R|^{\frac{p}{p-1}(\frac{p-1}{p})}\Vert \nabla(u-h) \Vert^{p}_{L^{p}}\\
& + & \dfrac{1}{4}\Vert \nabla(u-h) \Vert_{L^{p}} \\ \nonumber
& = & C(p)(\Vert f \Vert_{L^{n}})^{\frac{p}{p-1}}|B_R| + \dfrac{1}{4}\Vert \nabla(u-h) \Vert_{L^{p}}.
\end{eqnarray}
Taking into account \eqref{t 5.1} and replacing \eqref{ineq.p5.1} and \eqref{ineq.p5.3} into \eqref{ineq.minimum 1} we reach
\begin{eqnarray}\label{ineq.endlemma 2}
 \int_{B_r}|\nabla u-(\nabla u)_r|^pdX & \leq & \nonumber C(n,p)\left (\dfrac{r}{R} \right )^{n+p\alpha_{p}}  \int_{B_R} \left |\nabla u(X) - \left (\nabla u \right)_R \right |^p dX \\ \nonumber
& + & C(n,p)\big[C(\lambda_+,\lambda_-)\big]|B_R| + C(n,p)C(n,p,\lambda_+,\lambda_-, \Vert f \Vert_{L^{n}})|B_R| \\ \nonumber
& \leq & C\left (\dfrac{r}{R} \right )^{n+p\alpha_{p}}  \int_{B_R} \left |\nabla u(X) - \left (\nabla u \right)_R \right |^p dX + CR^{n}.
\end{eqnarray}
where $C=C(n,p,\lambda_{+}, \lambda_{-},\Vert f \Vert_{L^{n}})$ is a positive constant. In view of Lemma \ref{elem.ineq} we obtain
\begin{equation}\label{eq end thm 2}
 \int_{B_r(X_0)}|\nabla u-(\nabla u)_r|^pdX \leq C r^n,
\end{equation}
which shows that the gradient of $u$ lies in $\text{BMO}$ space and for any fixed subdomain $\Omega' \Subset \Omega$, there holds
$$
    \|\nabla u\|_{\text{BMO}(\Omega')} \le C(\Omega', n,p,\lambda_{+}, \lambda_{-},\Vert f \Vert_{L^{n}}).
$$
From Fefferman-Stein BMO Characterization Theorem, see \cite{FS}, there exist vector fields $\Gamma_0, \Gamma_1, \cdots \Gamma_n \in L^\infty(\Omega')$, such that
$$
    \nabla u(X) = \Gamma_0(X) + \sum\limits_{i=1}^n \mathcal{R}_j(\Gamma_j),
$$
where $\mathcal{R}_j$ denotes the classical Riesz transform,
$$
    \mathcal{R}_j(f) :=   f* K_j \quad \text{for } \quad K_j(X_j) := \dfrac{c_n X_j}{|X|^{n+1}}.
$$
It now follows by a similar reasoning employed in the Appendix of \cite{KNV} that
$$
    |\nabla u(X)| \le -\log |X-X_0|, \quad \text{for } X\in B_\rho(X_0), ~\rho \ll 1.
$$
Finally, by Morrey's type estimate, we obtain, for $s>n$,
$$
    \begin{array}{lll}
         \displaystyle |u(X) - u(X_0)| &\le& \displaystyle C|X-X_0|^{1-\frac{n}{r}} \cdot \left (\int_{B_r(X_0)} |\nabla u(Z)|^s dZ \right )^{1/p} \vspace{0.1cm} \\
            &\le&  \displaystyle C|X-X_0|^{1-\frac{n}{s}}   \left (\int_0^{|X-X_0|} |\log Z|^s \cdot |Z|^{n-1} dZ \right )^{1/s} \vspace{0.1cm} \\
            & \le &  \displaystyle C |X-X_0| \cdot \left | \log |X-X_0| \right |,
    \end{array}
$$
and the proof of Theorem \ref{thm 2} is  concluded. \qed

%%%%%%%%%%%%%%%%%%%%%%
\section{Lower gradient bounds} \label{S. nondeg}
%%%%%%%%%%%%%%%%%%%%%%

From this Section on, we aim towards gradient estimates to
minimizers of heterogeneous $p$-jet flow functional \eqref{cav
functional}. We remark once more that even for equations with no
free boundaries, say $\lambda_{-} = \lambda_{+}$, it is not possible
to obtain pointwise control of the gradient of $u_0$, under the
borderline condition $f \in L^n$. In this case, as proven in Theorem
\ref{thm 2}, the best control available is of logarithm order.
Therefore, from this Section on,  we shall assume the source
function $f(X)$, appearing in functional \eqref{cav functional} is
$q$-integrable, for $q > n$. Under such natural hypothesis, our next
Theorem shows that $u^{+}_0$ grows linearly away from the free
boundary $\mathfrak{F}^{+}:=
\partial \{ u > 0 \} \cap \Omega$.

\begin{theorem}\label{t.growth} Let $u_0$ be a local minimizer to $\mathcal{J}_0$, with $f\in L^q(\Omega)$, $q > n$, $\Omega'\Subset
\Omega$ and $X_0 \in \{u_0 > 0 \} \cap \Omega'$. There exists a
constant $c_0 > 0$ depending only on n, $p$, $\lambda_{+}$ and $\|f\|_{L^q(\Omega)}$  such that
$$
    u(X_0) \ge c_0 \dist(X_0, \mathfrak{F}^{+}).
$$
\end{theorem}
\begin{proof} Let us fix $X_0 \in \{u_0 > 0 \} \cap \Omega'$. It suffices to show such estimate for points $X_0 \in \{u_0 > 0 \} \cap \Omega'$ such that
$$
    0 < \dist(X_0, \mathfrak{F}^{+})  \ll \delta(n, p, \lambda_{+}, \|f\|_{L^q(\Omega)}),
$$
for $\delta(n, p, \lambda_{+}, \|f\|_{L^q(\Omega)})$ to be regulated \textit{a posteriori}. Let us denote $d := \dist(X_0, \mathfrak{F}^{+})$ and if we define
$$
    v(X) := \dfrac{1}{d} u_0(X_0 + d X),
$$
one easily verifies that $v$ is a local minimizer to
$$
    \mathcal{J}^d_0(\xi) := \int_{B_1}  \left ( |\nabla \xi|^p + \lambda_{+} \chi_{\{\xi > 0 \}} + d\cdot f(X_0 + d\cdot X) \cdot \xi(X) \right )  dX,
$$
in $W^{1,p}_0(B_1) + v$. The thesis of Theorem \ref{t.growth} is equivalent to proving that $v(0)$ is universally bounded away from zero. Clearly $v \ge 0$ in $B_1$. By Harnack inequality (see Theorem \ref{thm Harnack}), we have
\begin{equation}\label{harnack proof LG}
    \begin{array}{lll}
            v(X) &\le& C(n,p) \left \{ v(0) + \|d\cdot f(X_0 + d\cdot X) \|_{L^q(B_1)}^{\frac{1}{p-1}}  \right \} \\
             &\le& C(n,p) \left \{ v(0) + \left ( d^{1 - \frac{n}{q}}  \cdot \|   f  \|_q \right )^{\frac{1}{p-1}}  \right \} \quad \forall X \in B_{3/5}.
    \end{array}
\end{equation}
In the sequel, we choose a nonnegative, smooth radially symmetric cut-off function $\psi$ satisfying
$$
    \phi \equiv 0 \text{ in } B_{1/10} \quad \text{and} \quad \phi \equiv 1 \text{ in } B_1 \setminus B_{1/2}
$$
and define the test function $g$ in $B_1$
by
$$
    g(X) := \min \left \{v, C(n,p) \left \{ v(0) +  \left ( d^{1 - \frac{n}{q}}  \cdot \|   f  \|_q \right )^{\frac{1}{p-1}}  \right \}  \cdot \psi(X)  \right \}.
$$
Notice that $g \in W^{1,p}$ and from Harnack inequality, estimate
\eqref{harnack proof LG}, $g$ agrees with  $v$ in $B_1\setminus
B_{1/2}$. Let us label the set
$$
    B_{1/2} \supset  \Pi := \left \{ Y \in B_{1/2} : C(n,p) \left \{ v(0) +  \left ( d^{1 - \frac{n}{q}}  \cdot \|   f  \|_q \right )^{\frac{1}{p-1}}   \right \}  \cdot \psi(Y )   < v(Y) \right \} \supset B_{1/10}.
$$
From the minimality of $v$, we estimate
\begin{equation}\label{nondeg Eq00}
    \int_{\Pi} \lambda_{+}\left( 1 -\chi_{\{g > 0\}} \right ) + d\cdot f(X_0 + d\cdot X) \cdot \left [ v(X)-g(X) \right ] dX \le \int_{\Pi} \left ( |\nabla g|^p - |\nabla v|^p \right )dX.
\end{equation}
The right-hand side of \eqref{nondeg Eq00} is readily estimated as
\begin{equation}\label{nondeg Eq01}
    \begin{array}{lll}
            \displaystyle \int_{\Pi} \left ( |\nabla g|^p - |\nabla v|^p \right )dX &\le& \left [C(n,p) \left \{ v(0) +  \left ( d^{1 - \frac{n}{q}}  \cdot \|   f  \|_q \right )^{\frac{1}{p-1}}  \right \}  \cdot  \| \psi \|_\infty \right ]^p  \\
            &\le & C v(0)^p + C \left [ d^{1 - \frac{n}{q}}  \cdot \|   f  \|_q \right ]^{\frac{p}{p-1}}.
    \end{array}
\end{equation}
We now turn our efforts towards estimating the left-hand side of \eqref{nondeg Eq00} by below. Readily we obtain
\begin{equation}\label{nondeg Eq02}
    \begin{array}{lll}
            \displaystyle \int_{\Pi} \lambda_{+} \left( 1 -\chi_{\{g > 0\}} \right )dX & = & \displaystyle \int_{\Pi}  \lambda_{+} \chi_{\{g = 0 \}}  dX \\
             &\ge& \displaystyle \lambda_{+} \left | B_{1/10} \right |.
    \end{array}
\end{equation}
Invoking once more Harnack inequality \eqref{harnack proof LG} and the fact that $\Pi \subset B_{1/2}$, we estimate
\begin{equation}\label{nondeg Eq03}
            \displaystyle \int_{\Pi} d\cdot f(X_0 + d\cdot X) \cdot \left [ v(X)-g(X) \right ] dX \le C \left ( d^{1 - \frac{n}{q}}  \cdot \|   f  \|_q \right ) \cdot  \left \{ v(0) +  \left ( d^{1 - \frac{n}{q}}  \cdot \|   f  \|_q \right )^{\frac{1}{p-1}}  \right \}
\end{equation}
Combining \eqref{nondeg Eq01}, \eqref{nondeg Eq02} and  \eqref{nondeg Eq03} we reach
\begin{equation}\label{nondeg Eq04}
    C \left \{  v(0)^p  +  \left ( d^{1 - \frac{n}{q}}  \cdot \|   f  \|_q \right ) v(0) \right \} \ge \lambda_{+} \left | B_{1/10} \right | - C \left [ d^{1 - \frac{n}{q}}  \cdot \|   f  \|_q \right ]^{\frac{p}{p-1}}.
\end{equation}
Therefore, choosing  $0 < d  \le \delta(n, p, \lambda_{+}, \|f\|_{L^q(\Omega)}) \ll 1$, appropriately, we conclude
$$
    v(0) \ge c(n,p, \lambda_{+}, \|f\|_q) > 0,
$$
and the proof of Theorem \ref{t.growth} follows.
\end{proof}

\bigskip

Next we iterate linear growth established in Theorem \ref{t.growth} as we obtain a stronger non-degeneracy property for $u_0$ near the free boundary.

%%%%%%
\begin{theorem}\label{t.strong nondeg}
Let $u_0$ be a local minimizer to $\mathcal{J}_0$, with $f\in L^q(\Omega)$, $q > n$, $\Omega'\Subset
\Omega$ and $X_0 \in \{u_0 \ge 0 \} \cap \Omega'$.  There exists a constant $\underline{c} > 0$ depending  on
$n$, $p$, $\lambda_{+}$ and $\|f\|_{L^q(\Omega)}$, such that
$$
    \sup\limits_{B_r(X_0)} u_0^{+} \ge \underline{c} \cdot  r,
$$
for any $0 < r \le \dist(\partial \Omega', \partial \Omega)$.
\end{theorem}

\begin{proof} By continuity, it suffices to show $u_0$ is strongly non-degenerated, i.e., the thesis of Theorem \ref{t.strong nondeg}
holds within the positivity set
$$
    \Omega_0^{+} :=  \{u_0 > 0 \} \cap \Omega'.
$$
We will obtain such a result by iterating linear growth estimate. More precisely we will initially show that there exists a $\delta_0 > 0$ that depends only on $n$, $\Omega'$, $p$, $\lambda_{+}$ and $\|f\|_q$ such that if $X \in \{u_0 > 0 \} \cap \Omega'$, there holds
\begin{equation}\label{iteration_SN}
    \sup\limits_{B_{d(X)}(X_0)} u_0 \ge (1+\delta_0) u_0(X_0),
\end{equation}
where $d(X) := \text{dist}(X, \mathfrak{F}^{+}).$ In order to verify \eqref{iteration_SN}, let us assume, for the purpose of contradiction, that no such a $\delta_0$ exist. If so, it would be possible to find sequences $\delta_j = \text{o}(1)$ and $X_j \in \{u_0 > 0 \} \cap \Omega'$ satisfying
\begin{equation}\label{eq.iteration_SN01}
    \sup\limits_{B_{d_j}(X_j)} u_0 \le (1+\delta_j) u_0(X_j), \quad \text{ for } \quad d_j := \dist(X_j, \mathfrak{F}^{+}) = \text{o}(1).
\end{equation}
Let us consider the following normalized sequence of functions
$\varrho_j \colon B_1 \to \mathbb{R}$ defined by
$$
    \varrho_j(Z) := \dfrac{u_0\left ( X_j + d_j Z \right)}{u_0(X_j)}.
$$
Clearly, $\varrho_j(0) = 1$, and from \eqref{eq.iteration_SN01},
\begin{equation}\label{eq.iteration_SN02}
    0\le \varrho_j \le 1+\delta_j\text{ in } B_1.
\end{equation}
In addition,  $\varrho_j$  satisfies
\begin{equation}\label{eq.iteration_SN03}
    \Delta_p \varrho_j = \dfrac{d_j^p}{u_0(X_j)^{p-1}} \cdot f(X_j + d_j Z),
\end{equation}
in the distributional sense in $B_1$. Taking into account the
linear growth established in Theorem \ref{t.growth} and Equation
\eqref{eq.iteration_SN03}, we reach
\begin{equation}\label{eq.iteration_SN04}
    \left | \Delta_p \varrho_j \right | \le C d_j \cdot f(X_j + d_j Z), \quad B_1.
\end{equation}
From Harnack inequality, we deduce the sequence
$\{\varrho_j\}_{j\in \mathbb{N}}$ is locally equicontinuous in
$B_1$; thus, up to a subsequence,   $\varrho_j \to \varrho$
locally uniformly in $B_1$. Harnack inequality further reveals
that for any $|X|\le r <1$, there holds
\begin{equation}\label{eq.iteration_SN05}
     0 \le [1+\delta_j] - \varrho_j(X) \le C_r \left ( [ 1+\delta_j] - \varrho_j(0) + d_j^{1 - \frac{n}{q}}  \cdot \|   f  \|_q \right ) = C_r \cdot \text{o}(1).
\end{equation}
Letting $j\to \infty$ in the above estimate, we deduce the limiting blow up function $\varrho \equiv 1 \text{ in }  {B}_1$.

\par

\medskip

We now show that such a conclusion drives us to an inconsistency. To this end,  let $Y_j \in \mathfrak{F}^{+}$ be such that
$d_j = |X_j - Y_j|.$ Up to subsequence, there would hold
$$
    1 + \text{o}(1) = \varrho_j\left ( \dfrac{Y_j - X_j}{d_j} \right ) = 0,
$$
which clearly gives a contradiction for $j \gg 1$. We have  shown the validity of estimate \eqref{iteration_SN}.

\par
\medskip

To finish up the proof of Theorem \ref{t.strong nondeg}, we employ a Caffarelli's polygonal type of argument. That is, we construct a polygonal along which $u_0$ grows linearly.  Starting from $X_0 = X$, we find a sequence of points $\{X_{n}\}_{n \geq 0}$ such that:
\begin{enumerate}
\item $u_0(X_{n}) \geq (1+ \delta_{0})^n u_0(X_0)$
\item $|X_{n}-X_{n-1}| = \dist (X_{n-1}, \mathfrak{F}^{+})$
\item $u_0(X_n) - u_0(X_{n-1}) \ge c |X_n - X_{n-1}|$. In particular, $u_0(X_n) - u_0(X_0) \ge c|X_n - X_0|$.
\end{enumerate}

Since $u(x_{n}) \to \infty$ as $n \to \infty$ this process must be finite, that is, there exists
a last $X_{n_0}$ in the ball $B_{r}(X_0)$. For such a last point,
$$
    |X_{n_0} - X_0| \ge c_p r,
$$
Finally,
$$
    \sup\limits_{B_r(X)} u_0 \ge u_0(X_{n_0}) \ge u_0(X_0) +  c |X_n - X_0| \ge  \underline{c} \cdot r,
$$
and the proof is concluded.
\end{proof}

%%%%%%%%%%%%%%%%%%%%%%%%%%%%%
\section{Stability for free boundary problems} \label{S. stability}

In this section we show the stability of the family of free
boundary problems obtained by the minimization of the
non-differentiable functionals
\begin{equation}\label{eq. gamma}
    \mathcal{J}_\gamma(v):=\int_\Omega\left (|\nabla v|^p+ \lambda_+(v^+)^\gamma+\lambda_-(v^-)^\gamma +
    f(X)\cdot v \right )dX \longrightarrow \text{min},
\end{equation}
as $\gamma = \text{o}(1)$.  The ultimate goal of this section is
to show that any limit point $u_0$ of $\{u_\gamma\}_{\gamma =
\text{o}(1)}$  is a minimizer to the $p$-degenerate  cavitation
functional
\begin{equation}\label{cav functional S. stability}
    \mathcal{J}_0(v) := \int_\Omega\left (|\nabla
    v|^p+ \lambda_+ \chi_{\{v > 0 \}} +\lambda_{-}\chi_{\{v \le 0 \}} + f(X)\cdot  v \right
    )dX.
\end{equation}

Initially we show compactness of $\{u_\gamma\}_{0<\gamma\le 1}$ in
the $W^{1,p}$ topology.

\begin{proposition}\label{prop comp}
Let $u_{\gamma_j}$ be a sequence of minima to the functional
$\mathcal{J}_{\gamma_j}$, $f\in L^n$ and assume $u_{\gamma_j} \to
v$ a.e., $\gamma_j \to \gamma_0$. Then for any $0< E < \infty$, $u_{\gamma_j} \to v$ in the
$W_\text{loc}^{1,E}(\Omega)$ topology.
\end{proposition}

\begin{proof}
  It follows from Proposition \ref{PDE} and a.e. convergence that
  $\Delta_p u_{\gamma_j} \rightharpoonup \Delta_p v$ in the sense of measures. Thus, from
  truncation arguments, see for instance \cite{BM},
  \begin{equation}\label{comp eq01}
        \nabla u_{\gamma_j} \to \nabla v \text{ a.e. in } \Omega.
  \end{equation}
  From Theorem \ref{thm 2}, for any $\Omega' \Subset \Omega$, there
  exists a constant $C(n,p, \lambda_{+}, \lambda_{-}, \Omega', \|f\|_n)$,
  independent of $\gamma_j$, such that,
  \begin{equation}\label{comp eq02}
    \|\nabla u_{\gamma_j}\|_{\text{BMO}(\Omega')} \le C(n,p, \lambda_{+}, \lambda_{-}, \Omega', \|f\|_n).
  \end{equation}
  Thus, from John-Nirenberg's Theorem, for $1\le E < \infty $ fixed,
  \begin{equation}\label{comp eq03}
        \|\nabla u_{\gamma_j}\|_{L^{E+1}(\Omega')} \le C(n,p, \lambda_{+}, \lambda_{-}, \Omega', \|f\|_n).
  \end{equation}
    Finally combining \eqref{comp eq01}, \eqref{comp eq03} and
    classical arguments, see for instance, \cite{MT}, we deduce
    $$
        \nabla u_{\gamma_j} \to \nabla v \quad \text{ in }
        L^E(\Omega'),
    $$
    and the Proposition follows.
\end{proof}

\begin{theorem}\label{t.ACF funct} Let $u_0 := \lim\limits_{\gamma_j} u_{\gamma_j}$ as $\gamma_j \to 0$.  Then $u_0$ is a local minimizer of $\mathcal{J}_0$.
\end{theorem}

\begin{proof} Let $B_r$ be a ball in $\Omega$. Given an arbitrary $W^{1,p}$
function $\psi$ that agrees with $u_0$ on $\partial B_r$, we have
to show that
$$
    \mathcal{J}_0(B_r, u_0) \le \mathcal{J}_0(B_r, \psi).
$$
By density we may further assume that $\psi$ is bounded. Let us define the interpolated function
$$
    \psi_{\gamma_{j}, h} := \left \{ \begin{array}{ccc} u_0 + \dfrac{|X| -
    r}{h} \left ( u_{\gamma_{j}}  - u_0 \right ) &\text{in}& B_{r+h}
    \setminus B_r \\
    \psi & \text{in} & B_r. \end{array} \right.
$$
One simply verifies that
\begin{eqnarray}\label{AC eq 1}
| \nabla \psi_{\gamma_{j}, h}|^p & \le & C_{p} \left\lbrace
|\nabla u_0|^{p} + \dfrac{1}{h^{p}}|u_{\gamma_{j}} - u_0|^p +
|\nabla u_{\gamma_{j}} - \nabla u_{0}|^{p} \right\rbrace , \quad
\text{ in } B_{r+h} \setminus B_r.
\end{eqnarray}
In the above estimate, we have used the classical facts:
\begin{eqnarray}
\nabla \left( |X|\right)= \frac{X}{|X|} \ \ \ \mbox{and} \ \ \ \left( \dfrac{|X| - r}{h}\right)^{p} \leq 1 \ \ \ \ \  \mbox{in} \  B_{r+h} \setminus B_r.
\end{eqnarray}
By $L^\infty$ bounds, Theorem \ref{l.bound}, there exists a
constant $C_1 >0$, independent of $\gamma_{j}$, such that
$\|u_{\gamma_{j}} \|_\infty < C_1$. Thus, if we denote
$$
    H^{\pm}_{\gamma_{j}}(t) := (t^{\pm})^{\gamma_{j}},
$$
we have
\begin{equation}\label{AC eq 2}
 H^{\pm}_{\gamma_{j}}(\psi_{\gamma_{j}, h}) \le \left( 3C_{1}\right)^{\gamma_{j}}, \quad \text{in } B_{r+h} \setminus B_r,
\end{equation}
and
\begin{equation}\label{AC eq 2.1}
 H^{\pm}_{\gamma_{j}}(\psi_{\gamma_{j}, h}) \le \left( \Vert \psi \Vert_{L^{\infty}\left( B_r\right) }\right)^{\gamma_{j}}\chi_{\{u_{\gamma_{j}} \gtrless 0 \}}, \quad \text{in } B_r.
\end{equation}
We can estimate
\begin{eqnarray}\label{AC eq 3}
\mathcal{J}_{\gamma_{j}}(B_{r+h}, \psi_{\gamma_{j}, h}) &=& \nonumber \int_{B_{r+h} \setminus B_r}  |\nabla \psi_{\gamma_{j}, h}|^p + \lambda_{+} H^{+}_{\gamma_{j}}(\psi_{\gamma_{j}, h}) + \lambda_{-} H^{-}_{\gamma_{j}}(\psi_{\gamma_{j}, h}) dX \\
& + &  \int_{B_{r+h} \setminus B_r}f\left( X \right)\left[ u_{0} + \dfrac{|X| - r}{h} \left ( u_{\gamma_{j}}  - u_0 \right )\right]dX + \mathcal{J}_{\gamma_{j}}(B_r, \psi) \\ \nonumber
& \le & C_{p} \int_{B_{r+h} \setminus B_r}|\nabla u_{0}|^{p}dX + C_{p}\int_{B_{r+h} \setminus B_r}|\nabla u_{\gamma_{j}} - \nabla u_{0}|^{p}dX \\ \nonumber
& + & \left[ 2\lambda_{+} \left( 3C_{1}\right)^{\gamma_{j}} + 3C_{1}\right] C_{p}|B_{r+h} \setminus B_{r}| + \dfrac{C_{p}}{h^{p}} \int_{B_{r+h} \setminus B_r} |u_{\gamma_{j}} - u_0|^p dX \\ \nonumber
& + &  \mathcal{J}_0(B_r, \psi)+ \left( \Vert \psi \Vert_{L^{\infty}\left( B_r\right)}^{\gamma_{j}} - 1\right) \int_{B_r} \lambda_{+}\chi_{\{\psi >0 \}} + \lambda_{-}\chi_{\{\psi \leq 0 \}} dX \\ \nonumber
& + & |B_{r+h} \setminus B_r|^{1-\frac{1}{q}}\Vert f \Vert_{L^{q}\left( \Omega \right)}.
\end{eqnarray}
By pointwise convergence $u_{\gamma_{j}} \to u_0$ we have
\begin{equation}
\label{AC eq 3.1}
\lim_{j \rightarrow \infty} \displaystyle \int_{B_{r+h} \setminus B_r} |u_{\gamma_{j}} - u_0|^p dx = 0
\end{equation}
and by Proposition \ref{prop comp}
\begin{equation}
\label{AC eq 3.2}
\lim_{j \rightarrow \infty} \displaystyle \int_{B_{r+h} \setminus B_r} |\nabla u_{\gamma_{j}} - \nabla u_0|^p dx = 0.
\end{equation}
From the minimality property of $u_{\gamma_{j}}$,
\begin{equation}\label{AC eq 4}
    \mathcal{J}_{\gamma_{j}} (B_{r + h}, \psi_{\gamma_{j}, h}) \ge \mathcal{J}_{\gamma_{j}} (B_{r+h},
    u_{\gamma_{j}}) \ge \mathcal{J}_{\gamma_{j}} (B_r, u_{\gamma_{j}}) + \int_{B_{r+h}\setminus B_{r}}f(X)u_{\gamma_{j}}dX.
\end{equation}
Furthermore, it follows from Proposition \ref{prop comp}
\begin{equation}\label{AC eq 5}
    \int_{B_r} |\nabla u_0|^p dX  = \lim \limits_{j \to \infty}
    \int_{B_r} |\nabla u_{\gamma_{j}}|^p dX.
\end{equation}
By the pointwise convergence $u_{\gamma_{j}} \to u_0$ and Fatou's Lemma (see the proof of Theorem \ref{l.bound}),
we conclude
\begin{equation}\label{AC eq 6}
    \int_{B_r} \lambda_{+}\chi_{\{u_0 >0 \}} + \lambda_{-}\chi_{\{u_0 \leq 0 \}} dX \le \liminf\limits_{j \to \infty}
    \int_{B_r} \lambda_{+}\left( u_{\gamma_{j}}\right)^{\gamma_{j}} \chi_{\{u_{\gamma_{j}} >0 \}} + \lambda_{-}\left( u_{\gamma_{j}}\right)^{\gamma_{j}} \chi_{\{u_{\gamma_{j}} \leq 0
    \}}dX,
\end{equation}
and
\begin{equation}\label{AC eq 7}
\lim_{j \rightarrow \infty}\displaystyle \int_{B_r} f(X)u_{\gamma_{j}} dX = \lim_{j \rightarrow \infty}\displaystyle \int_{B_r} f(X)u_0 dX.
\end{equation}
Finally, combining \eqref{AC eq 3}--\eqref{AC eq 7} we reach
\begin{eqnarray}
\mathcal{J}_0(B_r, u_0) & \le & \nonumber \liminf\limits_{j \to \infty}
\mathcal{J}_{\gamma_{j}}(B_{r+h}, u_{\gamma_{j}}) \\
& \le & \mathcal{J}_0(B_r, \psi) + C_{p}\int_{B_{r+h} \setminus B_{r}} |\nabla u_{0}|^{p} dX  \\ \nonumber
& + & \left( 2\lambda_{+} + 3C_{1}\right)C_{p}|B_{r+h} \setminus B_{r}| + |B_{r+h} \setminus B_{r}|^{1-\frac{1}{q}}\Vert f \Vert_{L^{q}\left( \Omega \right)}.
\end{eqnarray}
Letting $h \rightarrow 0$, we finish the proof of Theorem \ref{t.ACF funct}.
\end{proof}

\bigskip

%\vspace{1cm}

\noindent \textsc{Eduardo V. Teixeira} \hfill \textsc{Raimundo Leit\~ao}\\
Universidade Federal do Cear\'a \hfill  Universidade Federal do Cear\'a \\
Campus of Pici - Bloco 914 \hfill Campus of Pici - Bloco 914 \\
Fortaleza - Cear\'a - Brazil  \hfill Fortaleza - Cear\'a - Brazil  \\
60.455-760 \hfill 60.455-760 \\
\texttt{teixeira@mat.ufc.br} \hfill \texttt{juniormatufc@yahoo.com.br}

\bigskip

\noindent \textsc{Olivaine S. de Queiroz} \\
Departamento de Matem\'atica\\
Universidade Estadual de Campinas -- IMECC\\
Rua S\'ergio Buarque de Holanda, 651\\
Campinas, SP, Brazil\\
CEP 13083-859\\
\texttt{olivaine@ime.unicamp.br}

\end{document}